\documentclass{amsart}
\usepackage{tikz}
\usetikzlibrary{arrows,matrix}
\usepackage[all]{xy}
\usepackage{amssymb}
\usepackage{caption,subcaption}
\usepackage{float}
\usepgflibrary{patterns}

\usepackage{hyperref}
\hypersetup{
    colorlinks=true, 
    linkcolor=blue,
    filecolor=magenta,
    urlcolor=cyan,
    citecolor=cyan}
\usepackage[nameinlink,capitalize]{cleveref}

\usepackage[ margin=2cm]{geometry}

\DeclareMathOperator{\Exp}{Exp}
\DeclareMathOperator{\inter}{int}

\def\bd{\mathbf{d}}
\def\be{\mathbf{e}}
\def\bv{\mathbf{v}}

\def\ba{\mathbf{a}}
\def\bb{\mathbf{b}}
\def\bc{\mathbf{c}}
\def\bw{\mathbf{w}}

\newcommand{\Spec}{\operatorname{Spec}}

\newcommand{\cF}{\mathcal{F}}

\def\fm{\mathfrak{m}}
\def\fp{\mathfrak{p}}

\def\NN{\mathbb{N}}
\def\ZZ{\mathbb{Z}}
\def\QQ{\mathbb{Q}}
\def\CC{\mathbb{C}}

\def\RR{\mathbb{R}}

\def\cD{\mathcal{D}}

\def\bu{\mathbf{u}}

\theoremstyle{Theorem}
\newtheorem{thm}{Theorem}[section]

\newtheorem*{mainthm}{Theorem}
\newtheorem*{maincor}{Corollary}
\newtheorem{cor}[thm]{Corollary}
\newtheorem{lem}[thm]{Lemma}

\theoremstyle{definition}

\newtheorem{dff}[thm]{Definition}
\newtheorem{xmp}[thm]{Example}
\newtheorem{rmk}[thm]{Remark}

\def\cone{\operatorname{cone}}
\def\Newt{\operatorname{Newt}}

\colorlet{DG}{green!50!black}
\colorlet{DB}{red!50!black}

\def\LEM[#1]{\footnote{ {\color{DG} LEM: #1  }  }}
\def\WDT[#1]{\footnote{ {\color{purple} WDT: #1  }  }}
\def\JV[#1]{\footnote{ {\color{DB} JV: #1  }  }}

\def\b0{\mathbf{0}}
\def\mon{\textnormal{mon}}
\def\int{\operatorname{int}}

\newcommand{\val}[1]{\nu(#1)}
\newcommand{\pval}[1]{\nu'(#1)}
\newcommand{\vpt}[1]{\omega(#1)}

\numberwithin{equation}{section}

\title{Differentially fixed ideals in toric varieties}
\author{ Lance Edward Miller, William D. Taylor, Janet Vassilev}

\begin{document}

\maketitle

\begin{abstract}
This article concerns monomial ideals fixed by differential operators of affine semi-group rings over $\mathbb{C}$. We give a complete characterization of when this happens. Perhaps surprisingly, every monomial ideal is fixed by an infinite set of homogeneous differential operators and is in fact determined by them. This opens up a new tool for studying monomial ideals. We explore applications of this to (mixed) multiplier ideals and other variants as well as give examples of detecting ideal membership in integrally closed powers and symbolic powers of squarefree monomial ideals.  
\end{abstract}

\section{Introduction}

In this article, we give a complete answer to the following question. What proper monomial ideals in a complex normal affine semigroup ring are fixed by a fixed homogeneous differential operator? That is to say, when does $\delta(I) = I$ for $\delta$ a homogeneous differential operator and $I$ a monomial ideal. This is inspired by a rough, but important, analogy between $p^{-e}$-linear maps which can be thought of as potential Frobenius splittings for rings of positive characteristic $p > 0$ and differential operators which of course can be considered in any characteristic, see for example \cite{BJNB19}. We note the main theorem of \cite{HSZ14} characterizes all ideals fixed by a fixed $p^{-e}$-linear map in the toric setting. In that case, there are finitely such fixed ideals \cite{BB09}. In contrast, we show every proper monomial ideal $I$ has infinitely many homogeneous differential operators fixing it. Moreover, the ideal $I$ is completely determined by the differential operators fixing it. Our results are also distinct from but in similar spirit to \cite{Ciu20} which where similar questions were addressed concerning ideals invariant under a fixed derivation.

\

Our approach is based on the Saito-Travis description of differential operators for such semigroup rings, \cite{ST01} and the structures subsequently used in \cite{BCK+,BCK+21}. This gives an explicit description of homogeneous differential operators $\delta$ of multidegree $\bd \in \ZZ^d$ on a semigroup ring $S$ defined by a cone of dimension $d$. Specifically, for each $\bd$, a homogeneous differential operator $\delta$ of degree $\bd$ has the form $\delta = x^{\bd} f$ where $x^{\bd}$ is a monomial in $d$ variables and $f$ is a polynomial differential operator of a prescribed form. 

\

Before stating the main theorem, we establish a bit of notation. For such an operator, a key role in the characterization of fixed ideals is played by the monomial vanishing locus. Specifically, for $\delta = x^{\bd} f$, as $f$ is a polynomial differential operator one can look at the vanishing locus of the polynomial defining it. The set of lattice points on which $f$ vanishes is the set  $V_{\mon}(f) = \{ \ba \in \ZZ^d \colon f(\ba) = \b0\}$. The geometric information in $V_{\mon}(f)$ that plays an essential role are the points in $V_{\mon}(f)$ which are furthest along rays in the direction of $-\bd$. To obtain these points, set $\val\ba := \inf\{t\in \RR \colon \ba + t\bd\in V_{\mon}(f)\}$. For $\delta$ to fix any ideals, we show it is necessary that $\val \ba >- \infty$ for all $\ba\in S$. When this condition is satisfied, we pick out the furthest points along these rays via the quantity $\pval\ba  :=  \max\{t\in [\val\ba,\val\ba + 1) \colon \ba + t\bd\in V_\mon(f)\}$ and consider the subset $V_{\mon}'(f)  := \{\ba + \pval\ba \bd \colon \ba\in \ZZ^d, \val\ba > -\infty\} \subset V_{\mon}(f).$ Utilizing this, we give a complete description of all ideals in $R$ are $\delta$-fixed for some homogeneous $\delta$, and what ideals are fixed for a given $\delta$ in terms of the exponent set $\Exp I = \{ \ba \colon x^{\ba} \in I\}$. We state these characterization as follows. 

\begin{mainthm}[cf. Theorem~\ref{thm:everyidealisdeltafixed} and Corollary~\ref{cor:interiordeltafix}]
Let  $R=\CC[S]$ be an affine semigroup ring defined by a cone $\sigma^\vee$, $I\subseteq R$ a proper monomial ideal, and $\bd\in\ZZ^d$ such that $-\bd\in S$. Set $\be\in\ZZ^d$ primitive such that $\bd = q\be$ for some $q\in \NN$.

\begin{itemize}
    \item There exists a homogeneous differential operator $\delta$ of degree $\bd$ such that $I$ is $\delta$-fixed.
    \item If $-\bd\in\inter \sigma^\vee$, then $\delta$ fixes a monomial ideal $I$ if and only if 
\begin{enumerate}
    \item $\val\ba>-\infty$ for all $\ba\in S$, 
    \item  for all $\ba\in V_{\mon}'(f)$ and $i=0,\ldots, q-1$, $\ba - i\be \in V_\mon(f)$,  and 
    \item for every $\ba,\bb\in V_{\mon}'(f)$, $\ba-\bb\notin S-\be$.  
\end{enumerate} In this case, $\delta$ fixes only one ideal, namely the ideal $I$ with $\Exp I = \{\ba\in S\mid \pval\ba\geq 0\}$.
\end{itemize}
\end{mainthm}

It is perhaps surprising but this characterization shows that every monomial ideal $I$ is completely determined by the non-empty set of differential operators $D_I := \{ \delta \colon \delta(I) = I\}$. This opens up a new line of study for monomial ideals. 

\

As an application, we consider a flexible but familiar setting where ideals are defined as the intersection of special polyhedral regions $P$. These regions simple from the point of view of combinatorial geometry. Specifically, any polyhedral region which is {\it $\sigma^\vee$-closed} in that they are closed under additive translation by $\sigma^\vee$ defines naturally an ideal and the usual correspondence between monomial ideals and their exponent sets shows all ideals are defined by such regions. We quite broadly work out explicit differential operators fixing such an ideal from the {\it face data}, i.e., hypersurfaces $\{ g_\tau \}$ defining $P$ together with integers $M_\tau$ which detect membership in $P$. This allows us to freely compute differential operations of various ideal theoretic constructions from their associated combinatioral geometric operations on polyhedra, a technique which is well-worn and pervades work on monomial ideals. In its most general form, the primary application is the following theorem.

\begin{mainthm}(cf. Theorem~\ref{thm:polyhedra}) Let $P_1,\ldots, P_n$ be $\sigma^\vee$-closed polyhedra with face data $\{(g_\tau,M_\tau)\}_{\tau\in\mathcal{S}_i}$ for $i=1,\ldots, n$ and $\b0\neq \bd\in-S$.
\begin{enumerate}
    \item Let $I=\langle x^{\ba } \mid \ba \in \bigcup P_i \rangle$ and  $f=\sum_{i=1}^n\prod_{\tau\in\mathcal{S}_i}(g_\tau(\theta),\lceil M_\tau-g_\tau(\bd)\rceil -1)!$.  For any $h\in (H_\bd)\cap(f)$, $\delta = x^\bd h(\theta)$ fixes $I$, and $x^\ba\in I$ if and only if $\delta(x^{\ba-\bd})\neq 0$.
    \item Let $J=\langle x^{\ba} \mid \ba\in \bigcup\inter P_i \rangle$ and  $f=\sum_{i=1}^n\prod_{\tau\in\mathcal{S}_i}(g_\tau(\theta),\lfloor M_\tau-g_\tau(\bd)\rfloor)!$.  For any $h\in (H_\bd)\cap(f)$, $\delta = x^\bd h(\theta)$ fixes $J$, and $x^\ba\in J$ if and only if $\delta(x^{\ba-\bd})\neq 0$.
\end{enumerate} 
\end{mainthm}

We draw out explicit consequences of this as many important constructions in algebraic geometry and commutative algebra considered on monomial ideals come down to essentially lattice point membership in such a polyhedral region. Examples of this include the Blickle-Howald description of multiplier ideals. In particular, we characterize all jumping numbers using the differential operators fixing an a multiplier ideal.  

\begin{maincor}(cf. Corollary~\ref{cor:lct}) Let $c\geq 0$ a real number and $I$ a monomial ideal of an affine semigroup ring $R$. Set $\bw = (1,1,\ldots,1)$, $X = \Spec R$, and $J(X,I^c) =\langle x^\ba \mid \ba + \bw \in c \cdot \inter \Newt I\rangle$ the multiplier ideal of the pair $(X,I^c)$. Set $(g_\tau,M_\tau)_{\tau \in \mathcal{F}}$ a face data for $\Newt I$.  
\begin{enumerate}
    \item The differential operator $\delta_c = x^{-\bw} H_{-\bw}(\theta)\prod_{\tau\in \mathcal{F}}(g_\tau(\theta), \lfloor cM_\tau\rfloor)!$ fixes $J_c$.
    \item The jumping numbers of $I$ are exactly $\left\{\min_{\tau\in\mathcal{F}}\left\{\frac{1}{M_\tau}g_\tau(\ba+\bw)\right\}\mid \ba\in S\right\}$.
    \item The log canonical threshold of $I$ is  $\min_{\tau\in\mathcal{F}}\left\{\frac{1}{M_\tau}g_\tau(\bw)\right\}$.
\end{enumerate}
\end{maincor}

We provide similar descriptions for mixed multiplier ideals, as well as ideal membership problems on integral closures of powers and symbolic powers in Section 5. All of these are specializations of the singular combinatorial geometric principal in last theorem but we work out samples of specific applications as well as the explicit differential operators at play. 

\

\noindent {\bf Acknowledgements:} We thank Mark Johnson and Paolo Mantero for supportive conversations and discussions on the material.

\section{Preliminaries and notation}

Throughout, $x$ will denote a list of variables $x_1,\ldots,x_d$ for $d \geq 1$. Thus we denote by $\CC[x] := \CC[x_1,\ldots,x_d]$ for the polynomial ring with variables $x$.  For a vector $\bv\in\ZZ^d$, $\bv=(v_1,\ldots,v_d)$ we denote by $x^{\bv}$ the monomial $x_1^{v_1}x_2^{v_2}\cdots x_d^{v_d}$. 

\

We also fix $\sigma = \cone(\bv_1,\ldots,\bv_n)$ a rational polyhedral cone for $\bv_1,\ldots,\bv_n \in \ZZ^d$ and set $R = \CC[S]$ where $S = \sigma^{\vee} \cap \ZZ^d$ which identifies $R = \CC[x^{\bv_1},\ldots,x^{\bv_n}]$ as a normal affine semigroup ring. Unless otherwise stated $I \subset R$ is a proper monomial ideal. We denote the exponent set by $\Exp I := \{ \ba \in S \colon x^{\ba} \in I \}$ and by $\int S$ the interior of $S$.    

\

We denote by $\cD$ the differential operators of $R$, for which we utilize the following description. Setting $\theta_{x_i} = x_i \partial_{x_i}$ and $\CC\langle\theta\rangle := \CC\langle\theta_{x_1},\ldots,\theta_{x_d}\rangle$, one may view $\cD$ as a subring of $\CC[ x^{\pm 1} ]\langle \theta\rangle$. We also utilize the descending factorial notation introduced in \cite{BCK+21,BCK+}, namely for $g \in \CC[x]$, $$(g,n)! = \prod_{i=0}^n (g-i) = 
\begin{cases}
g(g-1)(g-2)\cdots(g-n) & \text{if }n\geq 0 \\ 
1 & \text{if }n<0.
\end{cases}.$$ 

\

Throughout, for $f \in \CC[x]$ we abuse notation, without comment, by substituting either elements of the lattice $\ZZ^d$ or the tuple of variables $\theta$. Specifically, for $\ba \in \ZZ^d$, $\ba = (a_1,\ldots,a_d)$, we denote by $f(\ba) := f(a_1,\ldots,a_d)$. Similarly, we can evaluate such a function on $\theta$ to obtain a differential operator $f(\theta) = f(\theta_{x_1},\ldots,\theta_{x_d})$. Note that there is potential confusion in this abuse, if $\delta$ for example is a differential operator, then we also denote by $\delta(x^{\ba})$ for its action on a monomial. However, often we consider differential operators $\delta = x^{\bd} f$, so we are forced to consider $f$ both as its polynomial $\CC[x]$ and as the polynomial  differential operator $f(\theta)\in \CC\langle\theta\rangle$. For example, the notation $(f,n)!$ is consistent if we view $f$ as a polynomial or as a differential operator.

We also have need to consider the vanishing locus of a differential operator $f$ defined as follows. 

\begin{dff} For $f \in \CC[x]$, its {\it monomial vanishing locus} is the set $$V_{\mon}(f) := \{ \ba \in \ZZ^d \colon f(\ba) = \b0\} \subset \ZZ^d.$$ 
\end{dff}

\noindent The facets $\tau_1,\ldots,\tau_m$ of $\sigma^\vee$ determine unique homogeneous linear forms $h_i \in \CC[x]$ via their primitive support functions. For $\bd \in \ZZ^d$, the polynomials $H_{\bd}$, where $$H_{\bd} = \prod_{i} (h_i, h_i(-\bd)-1)!,$$ play an important role via the following theorem of Saito and Travis, \cite[Thm. 3.2.2]{ST01}.

\begin{thm}\label{thm:ST}[Saito-Traves]
There is a graded decomposition $$\cD(R) = \bigoplus_{\bd \in \ZZ^d} x^{\bd} ( H_{\bd}(\theta)).$$
Where the homogeneous elements are of the form $\delta = x^{\bd}f(\theta)$, where $f\in \CC[x]$, $H_\bd$ divides $f$, and for any monomial $x^a\in R$, we have that $\delta(x^\ba) = f(\ba)x^{\ba + \bd}$.
\end{thm}

\begin{xmp}
We take a moment to explore in an explicit example to clarify the introduced notation. Suppose $S$ is the semigroup generated by $x, xy, xy^2$ and $R = \CC[S] \subset \CC[x,y]$ so $\cD$ is a graded subring of $\CC[x^{\pm 1}, y^{\pm 1}]\langle \theta_x, \theta_y\rangle$. The primitive support functions are $h_1 = 2x - y$ and $h_2 = y$. This yields $h_1(\theta) = 2 \theta_x - \theta_y$ and $h_2(\theta) = \theta_y$. For each $\bd \in \ZZ^2$, we have a differential operator arising from $H_{\bd}$. For $\bd = (-2,-1)$, $h_1(-\bd) = h_1(2,1) = 4-1 = 3$ and $h_2(-\bd) = h_2(2,1) = 1$. Therefore, 
\begin{eqnarray*}
H_{\bd}(\theta) & = & \prod_{i} (h_i(\theta), h_i(-(-2,-1))-1)! \\
& = & (h_1(\theta),2)!(h_2(\theta),0)!\\
& = & (2\theta_x - \theta_y)(2\theta_x-\theta_y-1)(2\theta_x-\theta_y-2)(\theta_y).
\end{eqnarray*} It is also helpful to see the action of $h_1(\theta), h_2(\theta)$, and $H_{\bd}(\theta)$ on monomials. For example  $$h_1(\theta)(x^ay^b) = (2\theta_x - \theta_y)(x^ay^b) = 2 a x^ay^b - b x^ay^b = (2a - b) x^a y^b = h_1(a,b) x^ay^b.$$ Similarly, one may trivially verify
\begin{eqnarray*}
H_{\bd}(\theta)(x^ay^b) & = & (h_1,2)!(a,b)(h_2,0)!(a,b)x^ay^b \\ & = & (2a-b)(2a-b-1)(2a-b-2)bx^ay^b\\
\end{eqnarray*}

\end{xmp}

\noindent {\bf Notation:} Throughout the paper, unless otherwise stated $R = \CC[S]$ is a normal affine semigroup ring with $S = \sigma^{\vee} \cap \ZZ^d$ for rational polyhedral cone $\sigma$ and denote by $h_i$ the primitive support functions of $\sigma^{\vee}$. We always denote by $I$ a proper monomial ideal of $R$ and $\bd \in \ZZ^d$. We set $\mathcal{D} := \mathcal{D}(R)$ and $\delta = x^{\bd}f(\theta)$ a homogeneous differential operator of degree $\bd$ with $f \in (H_\bd) \subseteq \CC[x]$.

\section{Every monomial ideal is $\delta$-fixed}
Throughout, the aim of the paper is to study monomial ideals and what differential operators fix them in the following sense. 

\begin{dff}
For $\delta\in \cD$, an ideal $I \subset R$ is called {\it $\delta$-fixed} if $\delta(I) = I$ and {\it $\delta$-compatible} if $\delta(I) \subset I$. 
\end{dff} 

We characterize $\delta$-compatible ideals for homogeneous $\delta$ in terms of the exponent set.

\begin{thm}\label{thm:deltacompatiblecriterion} The ideal $I$ is $\delta$-compatible if and only if  $f(\ba) = 0$ for all $$\ba\in \Exp I\setminus (\Exp I - \bd).$$
\end{thm}

\begin{proof} For any $\ba\in \Exp I$, we have that $\delta(x^\ba) = f(\ba)x^{\bd + \ba}$. Therefore, $\delta(x^\ba)\in I$ if and only if $\bd+\ba\in \Exp I$ or $f(\ba) = 0$.  This is equivalent to the statement that if $\ba\notin \Exp I -\bd$, then $f(\ba)= 0$.
\end{proof}

Next we turn to characterizing when a given ideal is $\delta$-fixed for which we need the following lemma.

\begin{lem}\label{lem:imagedelta} The element $x^\ba\in R$ is in the image of $\delta$ if and only if $f(\ba - \bd)\neq 0$.
\end{lem}

\begin{proof} Since $\delta$ is homogeneous we may restrict our attention to monomials.  We have that $x^{\ba} = \delta(\alpha x^{\ba-\bd})$ for some $\alpha\in \CC$ if and only if 
$$x^{\ba} = \alpha f(\ba-\bd)x^{\ba}.$$
This is possible precisely when $f(\ba-\bd)\neq 0$.
\end{proof}

\begin{thm} \label{thm:fixedcriteria}  The ideal $I$ is $\delta$-fixed if and only if $-\bd\in S$ and $$\Exp I \cap V_\mon(f) = \Exp I \setminus (\Exp I-\bd).$$
\end{thm}

\begin{proof} By Theorem~\ref{thm:deltacompatiblecriterion}, $I$ is $\delta$-compatible if and only if $$\Exp I\setminus (\Exp I -\bd)\subseteq V_\mon(f).$$ To show $I\subseteq \delta(I)$,  it suffices to show that for every $\ba\in \Exp I$, $\ba-\bd\in \Exp I$ and $f(\ba-\bd)\neq 0$ by Lemma~\ref{lem:imagedelta}.  If $-\bd\notin S$, then for some $n\geq 1$,  $\ba -(n-1)\bd\in \Exp I$ but $\ba-n\bd \notin \Exp I$. Whence $x^{\ba-(n-1)\bd}\notin \delta(I)$.  Thus, $-\bd\in S$ is necessary.  It is also necessary that $f(\ba-\bd)\neq 0$ for all $\ba\in \Exp I$, or in other words, $f(\ba)\neq 0$ for every $\ba\in \Exp I - \bd$. It follows these are also sufficient again by Lemma~\ref{lem:imagedelta}.  Therefore $I=\delta(I)$ if and only if $-\bd\in S$ and $V_\mon(f)\cap (\Exp I-\bd) = \varnothing$, that is, $\Exp I \cap V_\mon(f) \subseteq \Exp I\setminus( \Exp I-\bd)$.
\end{proof}

Now we can show for fixed ideal $I$, there are many $\delta$ which fix $I$.

\begin{lem}\label{lem:nonzero}
For $\ba\in S$, $H_\bd(\ba) > 0$ if and only if $\ba\in -\bd + S$ and $H_{\bd}(\ba) = 0$ otherwise.
\end{lem}
\begin{proof} Recall $H_\bd = \prod_i (h_i,h_i(-\bd)-1)!$.
As $\ba \in S$, $H_\bd(\ba) > 0$ if and only if $h_i(\ba) - k > 0$ for all $0 \leq k \leq h_i(-\bd)-1$. This happens if and only if $h_i(\ba) - h_i(-\bd)+1 > 0$ which by linearity occurs if and only if $h_i(\ba + \bd) \geq 0$. Equivalently, this occurs when $\ba + \bd \in S$, or $\ba \in -\bd + S$.
\end{proof}

\begin{thm}\label{thm:everyidealisdeltafixed} Assume $-\bd\in S$.  There exists a homogeneous element $\delta\in \cD(R)$ of degree $\bd$ such that $I$ is $\delta$-fixed.
\end{thm}

\begin{proof} Let $I = (x^{\ba_1},\ldots, x^{\ba_m})$, so that $\Exp I =\bigcup_i (\ba_i+S)$.  Let $f = \sum_i H_{\bd-\ba_i}$, and note that $H_{\bd-\ba_i}(\bb) \geq 0$ for all $i$ and all $\bb\in S$.  Therefore, for any $\bb\in\Exp I$, we have that $f(\bb)\neq 0$ if and only if there exists $i$ such that $H_{\bd-\ba_i}(\bb) > 0$, which occurs exactly when $\bb\in \ba_i-\bd + S$ by Lemma~\ref{lem:nonzero}.  Thus $f(\bb)\neq 0$ if and only if 
$$\bb\in \bigcup_i (\ba_i-\bd +S) = -\bd +\bigcup_i(\ba_i+S) = -\bd + \Exp I.$$
Thus, we have that $V_\mon(f)\cap \Exp I = \Exp I \setminus (\Exp I-\bd)$.  Furthermore, since $H_{\bd}$ divides $H_{\bd-\ba_i}$ for each $i$, we have that $f\in (H_\bd)$, and so $\delta = x^\bd f(\theta)\in \cD(R)$, and $I$ is $\delta$-fixed.
\end{proof}

\begin{xmp} Let $R=\CC[x,xy,xy^2,xy^3]$.  Let 

Let 
    \begin{align*}
    \tau_1 &= \RR_{\geq 0}(1,0) &h_1(\theta) &= \theta_y\\
    \tau_2 &= \RR_{\geq 0}(0,1) &h_2(\theta)&= 3\theta_x-\theta_y\\
   \end{align*}
be the facets of the cone $\sigma^\vee$ of $R$.     Consider the ideal $I=(x^2y^4, x^5y)$ which is illustrated in Figure \ref{fig:RNC3ideal}.  We have labeled the lattice points representing the exponents of monomials which are contained in the ideal in gray (we have also shaded this region in gray) and those in the ring which are not members of the ideal in black.

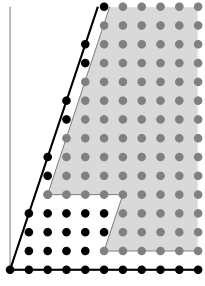
\begin{figure}[h]
  \centering
 \begin{tikzpicture}[scale=0.25]
\filldraw[fill=gray!30,draw=white] (10,1) -- (5,1) -- (6,4) --(2,4) -- (5.3,14) --(10,14) -- (10,1) -- cycle; 
\draw[black, thick] (0,0) -- (10.1 ,0);
\draw[gray] (0,0) -- (0,14);
\draw[black, thick] (0,0) -- (4.67,14);
\draw[gray] (2,4) -- (5.3,14);
\draw[gray] (2,4) -- (6,4);
\draw[gray] (5,1) -- (10,1);
\draw[gray] (5,1) -- (6,4);
\foreach \x in {1,2,3,4}{ \node[draw,circle,inner sep=1pt,black,fill] at (\x,3*\x) {};}
\foreach \x in {1,2,3,4,5}{ \node[draw,circle,inner sep=1pt,black,fill] at (\x,3*\x-1) {};}
\foreach \x in {2,3,4,5}{ \node[draw,circle,inner sep=1pt,gray,fill] at (\x,3*\x-2) {};}
\foreach \x in {0,1,...,10}{ \node[draw,circle,inner sep=1pt,black,fill] at (\x,0) {};}
\foreach \x in {1,2,3,4}{ \node[draw,circle,inner sep=1pt,black,fill] at (\x,1) {};}
\foreach \x in {5,6,...,10}{ \node[draw,circle,inner sep=1pt,gray,fill] at (\x,1) {};}
\foreach \x in {2,3,4,5}{ \node[draw,circle,inner sep=1pt,black,fill] at (\x,2) {};}
\foreach \x in {6,7,...,10}{ \node[draw,circle,inner sep=1pt,gray,fill] at (\x,2) {};}
\foreach \x in {2,3,4,5}{ \node[draw,circle,inner sep=1pt,black,fill] at (\x,3) {};}
\foreach \x in {6,7,...,10}{ \node[draw,circle,inner sep=1pt,gray,fill] at (\x,3) {};}
\foreach \x in {2,3,...,6}{ \node[draw,circle,inner sep=1pt,gray,fill] at (\x,4) {};}
\foreach \x in {7,8,...,10}{ \node[draw,circle,inner sep=1pt,gray,fill] at (\x,4) {};}
\foreach \x in {3,4,...,10}{ \node[draw,circle,inner sep=1pt,gray,fill] at (\x,5) {};}
\foreach \x in {3,4,...,10}{ \node[draw,circle,inner sep=1pt,gray,fill] at (\x,6) {};}
\foreach \x in {4,5,...,10}{ \node[draw,circle,inner sep=1pt,gray,fill] at (\x,7) {};}
\foreach \x in {4,5,...,10}{ \node[draw,circle,inner sep=1pt,gray,fill] at (\x,8) {};}
\foreach \x in {4,5,...,10}{ \node[draw,circle,inner sep=1pt,gray,fill] at (\x,9) {};}
\foreach \x in {5,6,...,10}{ \node[draw,circle,inner sep=1pt,gray,fill] at (\x,10) {};}
\foreach \x in {5,6,...,10}{ \node[draw,circle,inner sep=1pt,gray,fill] at (\x,11) {};}
\foreach \x in {5,6,...,10}{ \node[draw,circle,inner sep=1pt,gray,fill] at (\x,12) {};}
\foreach \x in {6,7,...,10}{ \node[draw,circle,inner sep=1pt,gray,fill] at (\x,13) {};}
\foreach \x in {6,7,...,10}{ \node[draw,circle,inner sep=1pt,gray,fill] at (\x,14) {};}
\end{tikzpicture}
\captionsetup{margin=0in,width=1.5in,font=small,justification=centering}
\caption{The ideal $I=(x^2y^4,x^5y)$}
\label{fig:RNC3ideal}
\end{figure}

Let $\bd=(-2,-1)$.  The differential operators in $\cD(R)$ of degree $(-2,-1)$ have the form $x^{-2}y^{-1}H_{(-2,-1)}(\theta)g(\theta)$ where $H_{(-2,-1)}(\theta)=(3\theta_x-\theta_y, 4)!\theta_y$.  In Figure \ref{fig:RNC3dopsex234} on the left, the lattice points in the solid gray region form the subset of $\Exp I$ whose images under any differential operator of degree $(-2,-1)$ also stay in $I$.  The lattice points in the vertically lined region are annihilated by any differential operator of degree $(-2,-1)$.  The lattice points in the crosshatched region represent monomials which lie in $I$, whose images lie outside of $I$ when acted upon by $x^{-2}y^{-1}H_{(-2,-1)}(\theta)$.  The middle and right images show similar regions for the homogeneous weights $\bd = (-4,-6)$ and $\bd = (-3,-8)$ respectively.

\begin{figure}[h]
 \centering
 \begin{tikzpicture}[scale=0.25]
 \filldraw[fill=gray!30,draw=white] (10,1) -- (5,1) -- (6,4) --(2,4) -- (5.3,14) --(10,14) -- (10,1) -- cycle; 
\pgfsetfillpattern{vertical lines}{gray}
\filldraw[
draw=white] (2,4) -- (3,4) -- (6.3, 14) -- (5.3,14) -- (2,4) -- cycle;
\pgfsetfillpattern{crosshatch}{black!60}
\filldraw[
draw=white] (3,4) -- (6,4) -- (5,1) -- (10,1) -- (10,2) -- (7,2) -- (8,5) --(4,5) -- (7,14) -- (6.3,14)-- (3,4) -- cycle;
\draw[black, thick] (0,0) -- (10,0);
\draw[gray] (0,0) -- (0,14.5);
\draw[black, thick] (0,0) -- (4.67,14);
\draw[gray] (4,5) -- (7,14);
\draw[gray] (4,5) -- (8,5);
\draw[gray] (7,2) -- (10,2);
\draw[gray] (7,2) -- (8,5);
\draw[gray] (2,4) -- (5.3,14);
\draw[gray] (2,4) -- (6,4);
\draw[gray] (5,1) -- (10,1);
\draw[gray] (5,1) -- (6,4);
\draw[gray] (2,1) -- (5,1);
\draw[gray] (3,0) -- (10,0);
\draw[gray] (2,1) -- (6.3,14);
\foreach \x in {1,2,3,4}{ \node[draw,circle,inner sep=1pt,black,fill] at (\x,3*\x) {};}
\foreach \x in {1,2,3,4,5}{ \node[draw,circle,inner sep=1pt,black,fill] at (\x,3*\x-1) {};}
\foreach \x in {2,3,4,5}{ \node[draw,circle,inner sep=1pt,gray,fill] at (\x,3*\x-2) {};}
\foreach \x in {0,1,...,10}{ \node[draw,circle,inner sep=1pt,black,fill] at (\x,0) {};}
\foreach \x in {1,2,3,4}{ \node[draw,circle,inner sep=1pt,black,fill] at (\x,1) {};}
\foreach \x in {5,6,...,10}{ \node[draw,circle,inner sep=1pt,gray,fill] at (\x,1) {};}
\foreach \x in {2,3,4,5}{ \node[draw,circle,inner sep=1pt,black,fill] at (\x,2) {};}
\foreach \x in {6,7,...,10}{ \node[draw,circle,inner sep=1pt,gray,fill] at (\x,2) {};}
\foreach \x in {2,3,4,5}{ \node[draw,circle,inner sep=1pt,black,fill] at (\x,3) {};}
\foreach \x in {6,7,...,10}{ \node[draw,circle,inner sep=1pt,gray,fill] at (\x,3) {};}
\foreach \x in {3,4,5,6}{ \node[draw,circle,inner sep=1pt,gray,fill] at (\x,4) {};}
\foreach \x in {7,8,9,10}{ \node[draw,circle,inner sep=1pt,gray,fill] at (\x,4) {};}
\foreach \x in {3,4,...,10}{ \node[draw,circle,inner sep=1pt,gray,fill] at (\x,5) {};}
\foreach \x in {3,4,...,10}{ \node[draw,circle,inner sep=1pt,gray,fill] at (\x,6) {};}
\foreach \x in {4,5,...,10}{ \node[draw,circle,inner sep=1pt,gray,fill] at (\x,7) {};}
\foreach \x in {4,5,...,10}{ \node[draw,circle,inner sep=1pt,gray,fill] at (\x,8) {};}
\foreach \x in {4,5,...,10}{ \node[draw,circle,inner sep=1pt,gray,fill] at (\x,9) {};}
\foreach \x in {5,6,...,10}{ \node[draw,circle,inner sep=1pt,gray,fill] at (\x,10) {};}
\foreach \x in {5,6,...,10}{ \node[draw,circle,inner sep=1pt,gray,fill] at (\x,11) {};}
\foreach \x in {5,6,...,10}{ \node[draw,circle,inner sep=1pt,gray,fill] at (\x,12) {};}
\foreach \x in {6,7,...,10}{ \node[draw,circle,inner sep=1pt,gray,fill] at (\x,13) {};}
\foreach \x in {6,7,...,10}{ \node[draw,circle,inner sep=1pt,gray,fill] at (\x,14) {};}
\foreach \x in {3,4,5}{ \node[draw,circle,inner sep=1pt,gray,fill] at (\x,3*\x-3) {};}
\foreach \x in {3,4,5,6}{ \node[draw,circle,inner sep=1pt,gray,fill] at (\x,3*\x-4) {};}
\foreach \x in {3,4,5,6}{ \node[draw,circle,inner sep=1pt,gray,fill] at (\x,3*\x-5) {};}
\foreach \x in {4,5,6}{ \node[draw,circle,inner sep=1pt,gray,fill] at (\x,3*\x-6) {};}
\foreach \x in {6,7}{ \node[draw,circle,inner sep=1pt,gray,fill] at (\x,3) {};}
\foreach \x in {6}{ \node[draw,circle,inner sep=1pt,gray,fill] at (\x,2) {};}
\end{tikzpicture}
\hspace{1cm}
 \begin{tikzpicture}[scale=0.25]
  \filldraw[fill=gray!30,draw=white] (10,1) -- (5,1) -- (6,4) --(2,4) -- (5.3,14) --(10,14) -- (10,1) -- cycle; 
\pgfsetfillpattern{vertical lines}{gray}
  \filldraw[
draw=white] (10,6) -- (10,1) -- (5,1) -- (6,4) -- (2,4) -- (5.3,14)-- (6.67,14)--(4,6)--(10,6)--cycle;
\pgfsetfillpattern{crosshatch}{black!60}
\filldraw[
draw=white] (10,6) -- (10,7) -- (9,7) -- (10,10) -- (6,10) -- (7.3,14) -- (6.6,14) --(4,6) -- (10,6) -- cycle;
\draw[black, thick] (0,0) -- (10,0);
\draw[gray] (0,0) -- (0,14.5);
\draw[black, thick] (0,0) -- (5,15);
\draw[gray] (2,4) -- (5.3,14);
\draw[gray] (2,4) -- (6,4);
\draw[gray] (5,1) -- (10,1);
\draw[gray] (5,1) -- (6,4);
\draw[gray] (4,6) -- (10,6);
\draw[gray] (4,6) -- (6.6,14);
\draw[gray] (6,10)--(10,10);
\draw[gray] (6,10)--(7.3,14);
\draw[gray] (9,7) -- (10,10);
\draw[gray] (9,7) -- (10,7);
\foreach \x in {1,2,3,4}{ \node[draw,circle,inner sep=1pt,black,fill] at (\x,3*\x) {};}
\foreach \x in {1,2,3,4,5}{ \node[draw,circle,inner sep=1pt,black,fill] at (\x,3*\x-1) {};}
\foreach \x in {2,3,4,5}{ \node[draw,circle,inner sep=1pt,gray,fill] at (\x,3*\x-2) {};}
\foreach \x in {0,1,...,10}{ \node[draw,circle,inner sep=1pt,black,fill] at (\x,0) {};}
\foreach \x in {1,2,3,4}{ \node[draw,circle,inner sep=1pt,black,fill] at (\x,1) {};}
\foreach \x in {5,6,...,10}{ \node[draw,circle,inner sep=1pt,gray,fill] at (\x,1) {};}
\foreach \x in {2,3,4,5}{ \node[draw,circle,inner sep=1pt,black,fill] at (\x,2) {};}
\foreach \x in {6,7,...,10}{ \node[draw,circle,inner sep=1pt,gray,fill] at (\x,2) {};}
\foreach \x in {2,3,4,5}{ \node[draw,circle,inner sep=1pt,black,fill] at (\x,3) {};}
\foreach \x in {6,7,...,10}{ \node[draw,circle,inner sep=1pt,gray,fill] at (\x,3) {};}
\foreach \x in {3,4,...,6}{ \node[draw,circle,inner sep=1pt,gray,fill] at (\x,4) {};}
\foreach \x in {7,8,...,10}{ \node[draw,circle,inner sep=1pt,gray,fill] at (\x,4) {};}
\foreach \x in {3,4,...,10}{ \node[draw,circle,inner sep=1pt,gray,fill] at (\x,5) {};}
\foreach \x in {3,4,...,10}{ \node[draw,circle,inner sep=1pt,gray,fill] at (\x,6) {};}
\foreach \x in {4,5,...,10}{ \node[draw,circle,inner sep=1pt,gray,fill] at (\x,7) {};}
\foreach \x in {4,5,...,10}{ \node[draw,circle,inner sep=1pt,gray,fill] at (\x,8) {};}
\foreach \x in {4,5,...,10}{ \node[draw,circle,inner sep=1pt,gray,fill] at (\x,9) {};}
\foreach \x in {5,6,...,10}{ \node[draw,circle,inner sep=1pt,gray,fill] at (\x,10) {};}
\foreach \x in {5,6,...,10}{ \node[draw,circle,inner sep=1pt,gray,fill] at (\x,11) {};}
\foreach \x in {5,6,...,10}{ \node[draw,circle,inner sep=1pt,gray,fill] at (\x,12) {};}
\foreach \x in {6,7,...,10}{ \node[draw,circle,inner sep=1pt,gray,fill] at (\x,13) {};}
\foreach \x in {6,7,...,10}{ \node[draw,circle,inner sep=1pt,gray,fill] at (\x,14) {};}
\foreach \x in {3,4,5}{ \node[draw,circle,inner sep=1pt,gray,fill] at (\x,3*\x-3) {};}
\foreach \x in {3,4,5,6}{ \node[draw,circle,inner sep=1pt,gray,fill] at (\x,3*\x-4) {};}
\foreach \x in {3,4,5,6}{ \node[draw,circle,inner sep=1pt,gray,fill] at (\x,3*\x-5) {};}
\foreach \x in {4,5,6}{ \node[draw,circle,inner sep=1pt,gray,fill] at (\x,3*\x-6) {};}
\foreach \x in {5,6,7}{ \node[draw,circle,inner sep=1pt,gray,fill] at (\x,3*\x-7) {};}
\foreach \x in {5,6,7,8}{ \node[draw,circle,inner sep=1pt,gray,fill] at (\x,7) {};}
\foreach \x in {6,7,8,9}{ \node[draw,circle,inner sep=1pt,gray,fill] at (\x,8) {};}
\foreach \x in {6,7,8,9}{ \node[draw,circle,inner sep=1pt,gray,fill] at (\x,9) {};}
\foreach \x in {6,7}{ \node[draw,circle,inner sep=1pt,gray,fill] at (\x,3) {};}
\foreach \x in {6}{ \node[draw,circle,inner sep=1pt,gray,fill] at (\x,2) {};}
\end{tikzpicture}
\hspace{1cm}
\begin{tikzpicture}[scale=0.25]
   \filldraw[fill=gray!30,draw=white] (10,1) -- (5,1) -- (6,4) --(2,4) -- (5.3,14) --(10,14) -- (10,1) -- cycle; 
\pgfsetfillpattern{vertical lines}{gray}
  \filldraw[
draw=white] (10,8) -- (10,1) -- (5,1) -- (6,4) -- (2,4) -- (3.3,8)--(10,8)--cycle;
\pgfsetfillpattern{crosshatch}{black!60}
\filldraw[
draw=white] (10,8) -- (10,9) -- (8,9) -- (9,12) -- (5,12) -- (5.67,14) --(5.3,14) --(3.3,8)-- (10,8) -- cycle;
\draw[black, thick] (0,0) -- (10,0);
\draw[gray] (0,0) -- (0,14);
\draw[black, thick] (0,0) -- (5,15);
\draw[gray] (2,4) -- (5.3,14);
\draw[gray] (2,4) -- (3.3,8);
\draw[gray] (2,4) -- (6,4);
\draw[gray] (3.3,8) -- (5.3 ,14);
\draw[gray] (5,1) -- (10,1);
\draw[gray] (5,1) -- (6,4);
\draw[gray] (3,8) -- (5,14);
\draw[gray] (3,8) -- (10,8);
\draw[gray] (5,12)--(9,12);
\draw[gray] (5,12)--(5.67,14);
\draw[gray] (8,9) -- (9,12);
\draw[gray] (8,9) -- (10,9);
\foreach \x in {1,2,3,4}{ \node[draw,circle,inner sep=1pt,black,fill] at (\x,3*\x) {};}
\foreach \x in {1,2,3,4,5}{ \node[draw,circle,inner sep=1pt,black,fill] at (\x,3*\x-1) {};}
\foreach \x in {2,3,4,5}{ \node[draw,circle,inner sep=1pt,gray,fill] at (\x,3*\x-2) {};}
\foreach \x in {0,1,...,10}{ \node[draw,circle,inner sep=1pt,black,fill] at (\x,0) {};}
\foreach \x in {1,2,3,4}{ \node[draw,circle,inner sep=1pt,black,fill] at (\x,1) {};}
\foreach \x in {5,6,...,10}{ \node[draw,circle,inner sep=1pt,gray,fill] at (\x,1) {};}
\foreach \x in {2,3,4,5}{ \node[draw,circle,inner sep=1pt,black,fill] at (\x,2) {};}
\foreach \x in {6,7,...,10}{ \node[draw,circle,inner sep=1pt,gray,fill] at (\x,2) {};}
\foreach \x in {2,3,4,5}{ \node[draw,circle,inner sep=1pt,black,fill] at (\x,3) {};}
\foreach \x in {6,7,8,9,10}{ \node[draw,circle,inner sep=1pt,gray,fill] at (\x,3) {};}
\foreach \x in {2,3,...,6}{ \node[draw,circle,inner sep=1pt,gray,fill] at (\x,4) {};}
\foreach \x in {7,8,9,10}{ \node[draw,circle,inner sep=1pt,gray,fill] at (\x,4) {};}
\foreach \x in {3,4,...,10}{ \node[draw,circle,inner sep=1pt,gray,fill] at (\x,5) {};}
\foreach \x in {2}{ \node[draw,circle,inner sep=1pt,black,fill] at (\x,5) {};}
\foreach \x in {3,4,...,10}{ \node[draw,circle,inner sep=1pt,gray,fill] at (\x,6) {};}
\foreach \x in {4,5,...,10}{ \node[draw,circle,inner sep=1pt,gray,fill] at (\x,7) {};}
\foreach \x in {4,5,...,10}{ \node[draw,circle,inner sep=1pt,gray,fill] at (\x,8) {};}
\foreach \x in {4,5,...,7}{ \node[draw,circle,inner sep=1pt,gray,fill] at (\x,9) {};}
\foreach \x in {8,9,10}{ \node[draw,circle,inner sep=1pt,gray,fill] at (\x,9) {};}
\foreach \x in {5,6,...,8}{ \node[draw,circle,inner sep=1pt,gray,fill] at (\x,10) {};}
\foreach \x in {9,10}{ \node[draw,circle,inner sep=1pt,gray,fill] at (\x,10) {};}
\foreach \x in {5,6,7,8}{ \node[draw,circle,inner sep=1pt,gray,fill] at (\x,11) {};}
\foreach \x in {9,10}{ \node[draw,circle,inner sep=1pt,gray,fill] at (\x,11) {};}
\foreach \x in {5,6,...,10}{ \node[draw,circle,inner sep=1pt,gray,fill] at (\x,12) {};}
\foreach \x in {6,7,...,10}{ \node[draw,circle,inner sep=1pt,gray,fill] at (\x,13) {};}
\foreach \x in {6,7,...,10}{ \node[draw,circle,inner sep=1pt,gray,fill] at (\x,14) {};}
\end{tikzpicture}
\captionsetup{margin=0in,width=5in,font=small,justification=centering}
\caption{Left: $\bd=(-2,-1)$, Middle: $\bd=(-4,6)$, Right: $\bd=(-3,-8)$. }
\label{fig:RNC3dopsex234}
\end{figure}
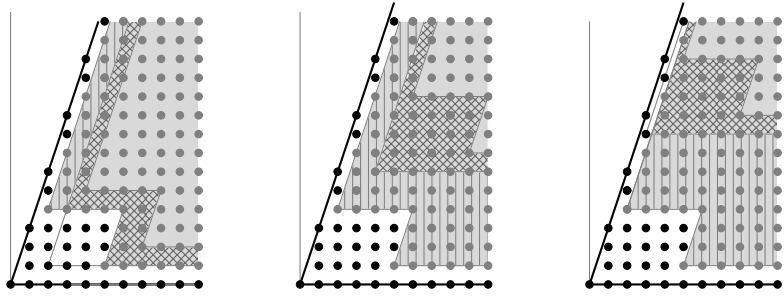

\end{xmp}

\begin{xmp} Let $R=\CC[x,y,z, xyz^{-1}]$.  Let 
\begin{align*}
    \tau_1 &= \RR_{\geq 0}(1,0,0) + \RR_{\geq 0}(0,0,1) & h_1(\theta) &= \theta_y\\
    \tau_2 &= \RR_{\geq 0}(0,0,1) + \RR_{\geq 0}(0,1,0) & h_2(\theta) &= \theta_x\\
    \tau_3 &= \RR_{\geq 0}(0,1,0) + \RR_{\geq 0}(1,1,-1) & h_3(\theta) &= \theta_x+\theta_z\\
    \tau_4 &= \RR_{\geq 0}(1,0,0) + \RR_{\geq 0}(1,1,-1) & h_4(\theta) &= \theta_y+\theta_z\\
\end{align*}
be the faces of the cone $\sigma^\vee$ of $R$.
Now let $\bd = (-2,0,-1)$, and let $I=(x^2,z,xy^2z^{-1})$.  We construct a differential operator of degree $\bd$ fixing $I$.

Now, we have that
\begin{align*} H_{\bd-(2,0,0)}(\theta) &= H_{(-4,0,-1)}\\
&=\prod_{i=1}^4(h_i(\theta),h_i(4,0,1)-1)!\\
&= (h_1(\theta),-1)!\cdot (h_2(\theta), 3)!\cdot (h_3(\theta),4)!\cdot (h_4(\theta),0)!\\
&=  (\theta_x,3)!(\theta_x+\theta_z, 4)!(\theta_y+\theta_z,0)!\\
H_{\bd-(0,0,1)}(\theta) &= H_{(-2,0,-2)}\\
&=\prod_{i=1}^4(h_i(\theta),h_i(2,0,2)-1)!\\
&= (h_1(\theta),-1)!\cdot (h_2(\theta), 1)!\cdot (h_3(\theta),3)!\cdot (h_4(\theta),1)!\\
&=  (\theta_x,1)!(\theta_x+\theta_z, 3)!(\theta_y+\theta_z,1)!\\
H_{\bd-(1,2,-1)}(\theta) &= H_{(-3,-2,0)}\\
&=\prod_{i=1}^4(h_i(\theta),h_i(3,2,0)-1)!\\
&= (h_1(\theta),1)!\cdot (h_2(\theta), 2)!\cdot (h_3(\theta),2)!\cdot (h_4(\theta),1)!\\
&=  (\theta_y,1)!(\theta_x,2)!(\theta_x+\theta_z, 2)!(\theta_y+\theta_z,1)!
\end{align*}
Therefore, a differential operator fixing $I$ is given by
$$\delta = x^{-2}z^{-1}\left(H_{(-4,0,-1)} + H_{(-2,0,-2)}+H_{(-3,-2,0)}\right)
$$
For example, we check that the generator $x^2$ of $I$ is in the image of $\delta$.  The preimage of the generator $x^2$ should be a multiple of $x^4z$.  We have that
$$\delta(x^4z) = x^{-2}z^{-1}\left(H_{(-4,0,-1)} + H_{(-2,0,-2)}+H_{(-3,-2,0)}\right)(x^4z) $$
$$= x^2 \left(H_{(-4,0,-1)}(4,0,1) + H_{(-2,0,-2)}(4,0,1)+H_{(-3,-2,0)}(4,0,1)\right)
$$
Now, we have that
\begin{align*}
    H_{(-4,0,-1)}(4,0,1) &= (4,3)!(5,4)!(1,0)! = 4\cdot 3\cdot 2\cdot 1\cdot 5\cdot 4\cdot 3\cdot 2\cdot 1\cdot 1 = 2880\\
    H_{(-2,0,-2)}(4,0,1) &= (4,1)!(5,3)!(1,1)! = 4 \cdot 3 \cdot 5\cdot 4\cdot 3\cdot 2\cdot 1\cdot 0 = 0\\
    H_{(-3,-2,0)}(4,0,1) &= (0,1)!(4,2)!(5,2)!(1,1)! = 0\cdot -1\cdot 4\cdot 3\cdot 2\cdot 5\cdot 4\cdot 3\cdot 1\cdot 0 = 0
\end{align*}
Therefore,
$$\delta(x^4z) = x^2\cdot (2880 + 0 + 0) = 2880x^2$$
Hence $x^2 = \delta\left(\frac{1}{2880} x^4z\right)$
\end{xmp}

\section{The Fixed Ideals of a Differential Operator}

We turn now to characterizing all ideals fixed by a given differential operator $\delta = x^{\bd}f(\theta)$. To this end, we are interested in the points in $V_{\mon}(f)$ which are furthest along rays in the direction of $-\bd$.

\begin{dff}
For every lattice point $\ba\in \ZZ^d$, set $$\val\ba := \inf\{t\in \RR\mid \ba + t\bd\in V_{\mon}(f)\}.$$  When $\val\ba$ is finite, set 
$$\pval\ba := \sup\{ t\in [\val \ba, \val \ba +1)\mid \ba + t\bd\in V_\mon(f)\}$$ and $\vpt\ba := \ba + \pval\ba \bd$.
\end{dff}

\noindent With the above notation, if $\val\ba$ is finite, then $\ba +\val\ba \bd$ is the point in $V_\mon(f)$ furthest along the ray through $\ba$ in the direction of $-\bd$.  Also in this case, $\vpt\ba$ is a point in $V_\mon(f)$ such that $\vpt\ba-n\bd\notin V_\mon(f)$ for any $n\geq 1$.  Note that if $\bd$ is primitive, then $\val\ba=\pval\ba$, since $\ba + t\bd$ will not be a lattice point if $t\notin \ZZ$. 

\

Suppose  $\delta = x^\bd f(\theta)$ is a differential operator on $R$ with $\b0\neq \bd\in -S$. We have that $\val\ba<\infty$ for any $\ba\in S$.  To see this, note if $\val\ba = \infty$, then $f(\ba+t\bd)\neq 0$ for all $t\in \RR$.  If we take $i$ such that $h_i(\bd)< 0$, then for large enough $m\in\NN$ we have that $h_i(\ba+m\bd) = h_i(\ba) + mh_i(\bd)<0$.  Therefore $\ba+m\bd\notin S$.  However, we also have that $x^{\ba+m\bd}$ is a nonzero multiple of $\delta^m(x^\ba)$ since $f(\ba + t\bd)\neq 0$ for all $t\in \RR$.  This is a contradiction since $\delta$ is an operator on $R$ and $x^{\ba +m\bd}\notin R$. 

\

We note a few more useful points about these functions.  If $\val \ba>0$, then $f(\ba)\neq 0$.  For any $u\in\RR$ we have that
\begin{align*} \val{\ba + u\bd} &=  \inf\{t\in \RR\mid \ba+u\bd + t\bd\in V_{\mon}(f)\}\\
&=\inf\{t+u\in \RR\mid \ba + (t+u)\bd\in V_{\mon}(f)\}-u\\
&= \val{\ba} - u.
\end{align*}
Furthermore, notice that if $\ba = \bb + u\bd$ for some $u\in \RR$, then $\vpt{\ba} = \vpt{\bb}$. By definition $\val \ba \leq \pval\ba <\val \ba +1$, which means that for any $\ba$, we have that
$$\val{\vpt\ba} = \val{\ba + \pval\ba\bd} = \val{\ba} -\pval\ba
$$
and so $\val{\vpt\ba}\in (-1,0]$. If $\be\in \ZZ^d$ is primitive such that $\bd=q\be$, then $\val \ba , \pval \ba\in \frac{1}{q}\ZZ$ for all $\ba\in \ZZ^d$.

\

In the next theorem, our aim is to ``cut out'' parts of the semigroup $S$ based on the values of the primitive support functions of the faces of $\sigma^\vee$.  We have that a lattice point $\ba\in \ZZ^d$ is in $S$ if and only if $h_i(\ba)\geq 0$ for all $i$. For any $\beta=(\beta_1,\ldots, \beta_n)\in \NN^n$, we consider the subset of the lattice points in $S$ whose values at the functions $h_i$ are at least the values $\beta_i$.  Let $W_\beta = \{\ba\in S\mid \textrm{for all } i, h_i(\ba)\geq \beta_i\}$.  Notice that the set $W_\beta$ is closed under the semigroup action, in other words $W_\beta$ is the exponent set of the ideal $I_\beta = (x^\ba\mid \ba\in W_\beta)$.

\

We also consider for a collection of vectors $\mathcal{B}\subseteq \NN^n$, $W_\mathcal{B} :=\bigcup_{\beta\in \mathcal{B}}W_\beta$.  Clearly $W_\mathcal{B}$ is also closed under the semigroup action. 

\begin{dff} For a vector $\bd\in -S$, we call a tuple $\beta\in\NN^n$ \emph{compatible with $\bd$} if for all $i$, either $\beta_i = 0$ or $h_i(\bd) = 0$.  

 A collection $\mathcal{B}\subseteq \NN^n$ is called \emph{compatible with $\bd$} if each $\beta\in \mathcal{B}$ is compatible with $\bd$.
\end{dff}

\noindent The tuples compatible with $\bd$ will be part of the data defining the fixed ideals of a differential operator $\delta$ of degree $\bd$.  Note that if $-\bd\in\inter \sigma^\vee$, then the only tuple compatible with $\bd$ is the zero tuple. 

\begin{lem} \label{lem:W_B is delta-compatible}If $\delta$ is a homogeneous operator of degree $\bd$, $\mathcal{B}\subseteq \NN^n$ is compatible with $\bd$, and $I=(x^{\ba}\mid \ba\in W_\mathcal{B})$, then $I$ is $\delta$-compatible.  In particular, if $\bb\in W_{\mathcal{B}}$ and $\val{\bb}$ is finite, then $\vpt\bb\in W_{\mathcal{B}}$.
\end{lem}

\begin{proof} If $x^\ba\in I$, then $\ba\in \Exp I = W_{\mathcal{B}}$.  Therefore there exists $\beta\in \mathcal{B}$ such that for all $i$, $h_i(\ba)\geq \beta_i$.  If $\delta(x^{\ba})\neq 0$, then $\ba+\bd\in S$.  For every $i$, if $h_i(\bd)\neq 0$, then $\beta_i=0$, and so $h_i(\ba+\bd)\geq \beta_i$.  If $h_i(\bd)=0$, then $h_i(\ba+\bd) = h_i(\ba) \geq \beta_i$.  Thus in either case $\ba+\bd\in W_\beta\subseteq W_{\mathcal{B}} =\Exp I$.  Therefore $\delta(x^\ba)\in I$, and so $I$ is $\delta$-compatible.

For the last statement, if $\pval\bb \leq 0$, then $\vpt\bb = \bb  + (-\pval\bb)(-\bd)\in \bb + S\subseteq W_\mathcal{B}$.  If $\pval\bb >0$, then since $\val{\vpt\bb} >-1$, we have that $f(\vpt\bb -m\bd)\neq 0$ for all $m\geq 1$, and so $x^{\vpt\bb}$ is a nonzero multiple of $\delta^{m}(x^{\vpt\bb - m\bd})$ for all $m\geq 1$.  Taking $m=\lceil\pval\bb\rceil$, and noting that 
$$\vpt \bb - \lceil \pval\bb\rceil \bd = \bb +\pval\bb \bd  - \lceil \pval\bb\rceil \bd = \bb + (\pval\bb -  \lceil \pval\bb\rceil)(-\bd)\in W_{\mathcal{B}}$$
we have that $x^{\vpt\bb}$ is a nonzero multiple of an element of $I$.  Since $I$ is $\delta$-compatible, $x^{\vpt\bb}\in I$, and so $\vpt\bb\in \Exp I = W_{\mathcal{B}}$.
\end{proof}

We now prove necessary and sufficient conditions for a differential operator $\delta = x^\bd f(\theta)$ to fix a monomial ideal. This is described using the following subset of $V_\mon (f)$: 
$$V_{\mon}'(f) := \{\vpt\ba \mid \ba\in \ZZ^d, \val\ba\text{ finite}\}.$$

\

\noindent {\bf Notation:} We also make the assumption for the rest of the section that $\bd\neq \mathbf{0}$ with $\bd=q\be$ with $\be$ primitive.

\begin{thm} If there exists a $\bd$-compatible subset $\mathcal{B}\subseteq \NN^n$ such that
\begin{enumerate}
    \item for all $\ba\in W_\mathcal{B}$, $\val{\ba}>-\infty$ 
    \item for all $\ba\in V_{\mon}'(f)\cap W_{\mathcal{B}}$ and $i=0,\ldots, q-1$, $\ba - i\be \in V_\mon(f)$, and 
    \item for all $\ba,\bb\in V_{\mon}'(f)\cap W_{\mathcal{B}}$, $\ba-\bb\notin S-\be$, 
\end{enumerate} then $\delta$ fixes the ideal $I=(x^{\ba}\mid \ba\in V_{\mon}'(f)\cap W_{\mathcal{B}})$. Furthermore $\Exp I = \{\bb\in W_{\mathcal{B}}\mid \pval\bb\geq 0\}$.
\end{thm}

\begin{proof} Let $I=(x^{\ba}\mid \ba\in V_{\mon}'(f)\cap W_{\mathcal{B}})$.  We prove the last statement first. Let $\bb\in \Exp I$ and $\ba\in V'_\mon(f)\cap W_{\mathcal{B}}$ such that $\bb\in \ba +S$. This implies that $\bb\in W_{\mathcal{B}}$.  By condition (1), $\vpt\bb$ exists and therefore $\vpt{\bb}\in W_{\mathcal{B}}\cap V'_\mon(f)$ by Lemma~\ref{lem:W_B is delta-compatible}.  If $\pval{\bb}<0$, then $\pval{\bb}\leq -\frac{1}{q}$, and so  $\pval\bb \bd +\be = (q\pval\bb +1)\be \in S$.  Therefore we have that
$$\vpt \bb - \ba  = \bb +\pval\bb \bd -\ba 
=(\bb - \ba) + (\pval \bb \bd +\be) - \be
\in S-\be, 
$$
a contradiction to condition (3).  Therefore $\pval\bb \geq 0$.  Hence $\Exp I\subseteq \{\bb\in W_{\mathcal{B}}\mid \pval \bb\geq 0\}$.

On the other hand, If $\bb\in W_\mathcal{B}$ and $\pval\bb\geq 0$, then $\vpt{\bb}\in W_\mathcal{B}\cap V'_{\mon}(f)\subseteq \Exp I$ and so $$\bb = \vpt\bb + (-\pval\bb)(-\bd)\in \vpt\bb + S\subseteq \Exp I.$$  Therefore $\{\bb\in W_{\mathcal{B}}\mid \pval\bb\geq 0\}\subseteq \Exp I$, and so we have equality.

\

Let $\bb\in \Exp I$.  If $\pval\bb\geq 1$, then $\val \bb>0$, and so $f(\bb)\neq 0$.  On the other hand, if $0\leq \pval\bb < 1$, then $\pval\bb = \frac{i}{q}$ for some $i\in\{0,\ldots, q-1\}$.  Therefore,
$$\bb = \vpt\bb - \pval\bb\bd = \vpt\ba - i\be\in V_\mon(f)
$$
by condition (2), and so $f(\bb)=0$.  Therefore $\bb\in V_{\mon}(f)$ if and only if $\pval\bb <1$, which occurs exactly when $\pval{\bb+\bd}<0$, that is, when $\bb+\bd\notin \Exp I$.  Therefore $\delta$ fixes $I$ by Theorem~\ref{thm:fixedcriteria}.
\end{proof}

\begin{thm} If the operator $\delta$ fixes a nonzero monomial ideal $I$, then there exists a $\bd$-compatible subset $\mathcal{B}\subseteq \NN^n$ such that
\begin{enumerate}
    \item for all $\ba\in W_\mathcal{B}$, $\val{\ba}>-\infty$
    \item  for all $\ba\in V_{\mon}'(f)\cap W_{\mathcal{B}}$ and $i=0,\ldots, q-1$, $\ba - i\be \in V_\mon(f)$,
    \item for all $\ba,\bb\in V_{\mon}'(f)\cap W_{\mathcal{B}}$, $\ba-\bb\notin S-\be$, and
    \item $I= (x^{\ba}\mid \ba\in V_{\mon}'(f)\cap W_{\mathcal{B}})$ and $\Exp I = \{\ba\in W_{\mathcal{B}}\mid \pval\ba\geq 0\}$
\end{enumerate}
\end{thm}

\begin{proof}  Suppose that $\delta$ fixes a monomial ideal $I=(x^{\bb_1},\ldots, x^{\bb_m})$.  For $i=1,\ldots, n$ and $j=1,\ldots, m$, let 
$\beta_{ij} = h_i(\bb_j)$ if $h_i(\bd)=0$ and $\beta_{ij}=0$ otherwise.  For $j=1,\ldots, m$, let
$\beta_j = (\beta_{1j},\beta_{2j},\ldots, \beta_{nj})\in \NN^n$, and then let $\mathcal{B} = \{\beta_1,\ldots, \beta_m\}$.  By construction, $\mathcal{B}$ is $\bd$-compatible.

(1) Let $\ba\in W_{\mathcal{B}}$, and choose $j$ such that $\ba\in W_{\beta_j}$.  Fix $i\in \{1,\ldots, n\}$.  If $h_i(\bd)<0$, then  for sufficiently large $s\in \NN$ we have that $h_i(\ba -s\bd - \bb_j) = h_i(\ba-\bb_j)-sh_i(\bd)\geq 0$.  If $h_i(\bd)=0$, then for all $s\in\NN$ we have that  $h_i(\ba -s\bd - \bb_j) = h_i(\ba)-h_i(\bb_j)\geq \beta_{ij}-\beta_{ij} =0$.  So, for some large $s$ we have that $\ba -s\bd -\bb_j\in S$, and so $\ba -s\bd\in \bb_j+S\subseteq \Exp I$.  Therefore, we have that for any $t\in\frac{1}{q}\ZZ$ with $t\leq -s$, $\ba +t\bd = \ba - s\bd + (t+s)\bd\in \Exp I$.  Since $I$ is $\delta$-fixed, this means that for all such $t$, $\ba + (t-1)\bd\notin V_\mon(f)$.  Therefore $\val\ba \geq -s-1>-\infty$.  

As a consequence to the last paragraph, if we take $\ba\in W_{\mathcal{B}}$ such that $\pval\ba = 0$, then the argument shows that there exists $s\in \NN$ such that $\ba - s\bd \in \Exp I$, and furthermore $f(\ba - t\bd)\neq 0$ for all $t\geq 1$.  Therefore, $x^{\ba}$ is a nonzero multiple of $\delta^s(x^{\ba-s\bd})\in \delta^s(I)=I$.  Therefore $V'_\mon(f)\cap W_{\mathcal{B}}\subseteq \Exp I$.

(2) Let $\ba\in V'_\mon(f)\cap W_{\mathcal{B}}\subseteq \Exp I$.  Suppose that for some $i\in \{0,\ldots, q-1\}$ we have that $\ba - i\be\notin V'_\mon(f)$.  Note this implies that $i\neq 0$.  Therefore $\delta(x^{\ba-i\be})\neq 0$, and so $\ba -i\be +\bd\in \Exp I$.  
Therefore, for all $j=i,i+1,\ldots, q-1$, we have that $\ba-j\be+\bd\in \Exp I$, and therefore $\ba-j\be\notin V'_\mon(f)$ since $I$ is $\delta$-fixed.  Hence $\ba-t\bd\notin V_\mon(f)$ for $t\geq \frac{i}{q}$, and so $\val\ba\geq -\frac{i-1}{q}$. Therefore, for any $0<t<1-\frac{i-1}{q} = \frac{q-i+1}{q}$, we must have that $\ba +t\bd\notin V_\mon(f)$ since otherwise $\pval\ba$ would be greater than 0.  Hence $\ba -i\be + \bd = \ba +\frac{q-i}{q}\bd\notin V_\mon(f)$, and so $\ba-i\be+2\bd\in \Exp I$.  However, this implies that $\ba+\bd\in \Exp I$, a contradiction since $\ba\in V_\mon(f)$.

(3) Fix $j\in \{0,\ldots, m\}$ and suppose that $\ba,\bb\in V'_\mon(f)\cap W_{\mathcal{B}}$, which implies that $\ba,\bb\in \Exp I$.  If $\ba-\bb\in S-\be$, then $\ba +\be\in \bb+S$, and so $\ba+\be\in \Exp I$.  However, by condition (2), $\ba - (q-1)\be\in V_\mon(f)$, and so $\ba + \be = \ba - (q-1)\be +\bd \notin \Exp I$, a contradiction.

(4) If $\ba\in \Exp I$, then there exists $\bb_j$ such that $\ba\in \bb_j+S$, and so $\ba\in W_{\beta_j}\subseteq W_{\mathcal{B}}$.  Furthermore, if $\pval\ba <0$, then $\vpt\ba\in V'_{\mon}\cap W_{\mathcal{B}}$ and $\ba -\lfloor \pval\ba\rfloor\bd  = \vpt\ba -i\be$ for some $i\in\{0,\ldots, q-1\}$, and so  $\ba -\lfloor \pval\ba\rfloor\bd\in V_\mon(f)$.  Therefore $x^\ba$ is not in the image of $\delta$, and so $\ba\notin \Exp I$, a contradiction.  Therefore $\Exp I\subseteq \{\ba\in W_{\mathcal{B}}\mid \pval{\ba}\geq 0\}$.  Hence,
$$\{\ba\in W_{\mathcal{B}}\mid \pval{\ba}\geq 0\}\subseteq (V'_\mon(f)\cap W_{\mathcal{B}})+S\subseteq \Exp I \subseteq\{\ba\in W_{\mathcal{B}}\mid \pval{\ba}\geq 0\} $$
and so we have equality throughout.
\end{proof}

\begin{cor}\label{cor:deltafix} Suppose $\bd\neq 0$. The operator $\delta$ fixes some nonzero monomial ideal if and only if there exists a $\bd$-compatible subset $\mathcal{B}\subseteq \NN^n$ such that
\begin{enumerate}
    \item For all $\ba\in W_\mathcal{B}$, $\val{\ba}>-\infty$
    \item  for all $\ba\in V_{\mon}'(f)\cap W_{\mathcal{B}}$ and $i=0,\ldots, q-1$, $\ba - i\be \in V_\mon(f)$,  and 
    \item for all $\ba,\bb\in V_{\mon}'(f)\cap W_{\mathcal{B}}$, $\ba-\bb\notin S-\be$.
\end{enumerate} In this case, $\delta$ fixes the ideal $I=(x^{\ba}\mid \ba\in V_{\mon}'(f)\cap W_{\mathcal{B}})$, and furthermore $\Exp I = \{\ba\in W_{\mathcal{B}}\mid \pval\ba\geq 0\}$.
\end{cor}

\begin{cor}\label{cor:interiordeltafix} Suppose $-\bd\in \inter\sigma^\vee$. The operator $\delta$ fixes some nonzero monomial ideal if and only if 
\begin{enumerate}
    \item For all $\ba\in S$, $\val{\ba}>-\infty$
    \item  for all $\ba\in V_{\mon}'(f)$ and $i=0,\ldots, q-1$, $\ba - i\be \in V_\mon(f)$,  and 
    \item for all $\ba,\bb\in V_{\mon}'(f)$, $\ba-\bb\notin S-\be$.  
\end{enumerate} In this case, $\delta$ fixes only the ideal $I=(x^{\ba}\mid \ba\in V_{\mon}'(f))$, and furthermore $\Exp I = \{\ba\in S\mid \pval\ba\geq 0\}$.
\end{cor}

\begin{proof} This corollary follows from Corollary~\ref{cor:deltafix} once we observe that if $-\bd\in\inter\sigma^\vee$, then $h_i(\bd)<0$ for all $i$.  Therefore the only $\bd$-compatible vector is $\b0\in\NN^n$, and $W_{\b0} = S$.
\end{proof}

\begin{cor} If $-\bd\in \inter\sigma^\vee$, then $\delta$ fixes at most one nonzero monomial ideal.
\end{cor}

\begin{cor}Let $\mathcal{B}\subseteq \NN^n$ be a $\bd$-compatible subset such that $\delta$ fixes the ideal $I_{\mathcal{B}}=(x^\ba\mid \ba\in V'_\mon(f)\cap W_\mathcal{B})$.  Now for any $\bd$-compatible $\mathcal{B}'\subseteq \NN^n$ with $W_{\mathcal{B}'}\subseteq W_{\mathcal{B}}$, $I_{\mathcal{B}'}=(x^\ba\mid \ba\in V'_\mon(f)\cap W_{\mathcal{B}'})$ is $\delta$-fixed.
\end{cor}

\begin{proof}
Clearly, we have that $W_{\mathcal{B}'}$ satisfies conditions (1)-(3) in Corollary~\ref{cor:deltafix}.
\end{proof}

\begin{cor} If $-\bd\in S\setminus \inter\sigma^\vee$ and $\delta$ fixes at least one monomial ideal, then $\delta$ fixes infinitely many monomial ideals.
\end{cor}

\begin{proof}  Let $\mathcal{B}\subseteq \NN^n$ be a $\bd$-compatible subset such that $\delta$ fixes the ideal $I_{\mathcal{B}}=(x^\ba\mid \ba\in V'_\mon(f)\cap W_\mathcal{B})$.  Let $\beta=(\beta_1,\beta_2,\ldots, \beta_n)\in \mathcal{B}$. Without loss of generality, suppose that $h_1(\bd)=0$.  Now for any $m\in\NN$, $\beta_m := (\beta_1+m,\beta_2,\ldots, \beta_n)$ is $\bd$-compatible and $W_{\beta_m}\subseteq W_{\mathcal{B}}$, so $I_m := (x^{\ba}\mid \ba\in V'_\mon(f)\cap W_{\beta_m})$ is $\delta$-fixed.

Each $I_m$ is nonzero.  Indeed, if we take any $\bb\in S\setminus \inter \sigma^\vee$, then $h_i(\bb)>0$ for all $i$, and so for sufficiently large $n\in\NN$ we have that $h_1(n\bb)\geq\beta_1 + m$ and $h_i(n\bb)\geq \beta_i$ for $i>1$.  Thus $n\bb\in W_{\beta_m}$, and so for any $\ba\in V'_\mon (f)$, we have that $x^{n\bb+\ba}\in I_m$.  However, we also have that $\bigcap_{m\in\NN} I_m = 0$, and so there are infinitely many distinct $I_m$ in the descending chain $I_1\supseteq I_2\supseteq \cdots$, each fixed by $\delta$.
\end{proof}

\section{Applications} 

We now review applications, which are all derived from Theorem~\ref{thm:polyhedra} and elementary combinatorial geometry. This is phrased in suitable abstraction to find interest beyond the applications we display. 

\

First, we review  a bit of the combinatorial geometry, for which one can see \cite{Grue03} for a deeper introduction. Among other things this will establish definitions of basic objects which sometimes vary slightly from source to source. By a {\it polyhedron} we mean a region of a Euclidean space of finite dimension defined as the intersection of finitely many half-spaces. This is most similar to $H$-polyhedra, as opposed to a $V$-polyehdra where one utilizes a definition in terms of vertices, in many cases these descriptions are equivalent. We avoid the term `polytope' which conflicts among standard sources.

\

Naturally, the $H$-description arises by preferring the {\it facets}, or maximal faces. We additionally find it convenient to use a related, slightly more flexible description as the model of computing or specifying a polyhedron $P$. We motivate this by the following natural set of data. Set $\cF$ the set of faces of $P$. Naturally associated to $\cF$ is a sequence $(g_{\tau},M_{\tau})_{\tau \in \cF}$ where $g_{\tau} \in  \ZZ[x_1,\ldots,x_d]$ is a linear form and $M_{\tau} \in \QQ$ with the property that $\bv \in P$ if and only if $g_{\tau}(\bv) \geq M_{\tau}$ for all $\tau \in \cF$. We have need to consider a slight generalization of this description to account for inefficiencies.

\begin{dff} Fix $P$ a polyhedron. 
A {\it face data} consists of a finite set $\mathcal{S}$ and a sequence $\{ (g_\tau, M_\tau)\}_{\tau \in \mathcal{S}}$ with $g_\tau \in \ZZ[x_1,\ldots,x_d]$ is a linear form, $M_\tau \in \QQ$ and a point $\bv \in P$ if and only if $g_{\tau}(\bv) \geq M_\tau$ for all $\tau$. 
\end{dff}

As described there is a canonical face data given the $H$-description of a polyhedron which motivates the name. However, as we have defined it there is not a unique face data describing a given polyhedron. In particular, a face data could have redundant information. Shortly, we will describe an algorithmic process for computing face data, but this will not be optimal in the sense that we do not guarantee it returns the canonical face data. However, in practice these redundancies cause no issue.  

\

By our definition, polyhedra are not necessarily bounded nor convex. However, at times we will prefer to work with convex polyhedra. Convex or not, a unifying property we insist holds identifies those polyhedra whose intersection with the lattice defining the fixed semigroup ring under consideration in an ideal.

\begin{dff}
A polyhedron $P$ is {\it $\sigma^{\vee}$-closed} proved $\mathbf{p} + \mathbf{a} \in P$ for all $\mathbf{p} \in P$ and $\mathbf{a} \in \sigma^{\vee}$.
\end{dff}

Notably, $\sigma^{\vee}$-closed polyhedra are never bounded. Any monomial ideal naturally has associated to it a canonical $\sigma^{\vee}$-closed polyhedron based on its exponent set. Also, by definition every $\sigma^{\vee}$-polyhedron defines an ideal. Obviously, this is not however a bijection as small perturbations in the faces of the polyhedron will often not change the associated ideal. Nonetheless, we refer to these structures as associated to each other. 

\

Now we come to the general form of our primary application which articulates lattice point membership in a union of polyhedra in terms of a differential operator fixing its associated ideal. 

\begin{thm}\label{thm:polyhedra} Let $P_1,\ldots, P_n$ be $\sigma^\vee$-closed polyhedra with face data $\{(g_\tau,M_\tau)\}_{\tau\in\mathcal{S}_i}$ for $i=1,\ldots, n$ and $\b0\neq \bd\in-S$.
\begin{enumerate}
    \item Let $I=\langle x^{\ba } \mid \ba \in \bigcup P_i \rangle$ and  $f=\sum_{i=1}^n\prod_{\tau\in\mathcal{S}_i}(g_\tau(\theta),\lceil M_\tau-g_\tau(\bd) -1\rceil)!$.  For any $h\in (H_\bd)\cap(f)$, $\delta = x^\bd h(\theta)$ fixes $I$, and $x^\ba\in I$ if and only if $\delta(x^{\ba-\bd})\neq 0$.
    \item Let $J=\langle x^{\ba} \mid \ba\in \bigcup\inter P_i \rangle$ and  $f=\sum_{i=1}^n\prod_{\tau\in\mathcal{S}_i}(g_\tau(\theta),\lfloor M_\tau-g_\tau(\bd)\rfloor)!$.  For any $h\in (H_\bd)\cap(f)$, $\delta = x^\bd h(\theta)$ fixes $J$, and $x^\ba\in J$ if and only if $\delta(x^{\ba-\bd})\neq 0$.
\end{enumerate} 
\end{thm}

\begin{proof} As the proof of both statements are similar, we only provide the proof of the first statement. We claim that $g_\tau(\ba)\geq 0$ for all $\tau\in \mathcal{F}$ and $\ba\in S$.  Let $\bb\in \Exp I$.  For any $m\in\NN$, $\bb + m\ba\in \Exp I\subseteq P$, and so $M_\tau\leq g_\tau(\bb + m\ba) = g_\tau(\bb) + m g_\tau(\ba)$.  Therefore $g_\tau(\ba)\geq (M_\tau - g_\tau(\bb))/m$ for all $m\in \NN$, and so $g_\tau(\ba)\geq 0$.

Therefore, for any face $\tau$, we have that $(g_\tau(\theta),\lceil M_\tau-g_\tau(\bd)\rceil -1)!(\ba)\geq 0$.  Hence $f(\ba)\neq 0$ if and only if there exists $1\leq i\leq n$ such that $(g_\tau(\ba),\lceil M_\tau-g_\tau(\bd)\rceil -1)!>0$ for all $\tau\in\mathcal{S}_i$.  This occurs exactly when $g_\tau(\ba)>\lceil M_\tau-g_\tau(\bd)\rceil-1$, which is equivalent to $g_\tau(\ba+\bd)\geq M_\tau$.  Thus, $f(\ba)\neq 0$ if and only if there exists $i$ such that $g_\tau(\ba+\bd)\in P_i$, i.e.\ $\ba+\bd\in \bigcup_i P_i = \Exp I$.  Therefore $V_\mon(f)\cap \Exp I = \Exp I \setminus (\Exp I-\bd)$, and so $\delta$ fixes $I$ by Theorem~\ref{thm:fixedcriteria}.  Also, we have that $f(\ba-\bd)=0$ for $\ba\notin \Exp I$, which proves the last statement as well.
\end{proof}

\begin{rmk} Theorem~\ref{thm:everyidealisdeltafixed} is a special case of Theorem~\ref{thm:polyhedra}.  This comes from the fact that every monomial ideal $I$ is of the form in the hypothesis of \ref{thm:polyhedra}.  Precisely, if $I=(x^{\ba_1},\ldots, x^{\ba_n})$, we may set $P_i = \ba_i + \sigma^\vee$ for each $i$.  Now each $P_i$ is a polyhedron, and $I=(x^\ba \mid \ba\in \bigcup_i P_i)$.  The improvement in Theorem~\ref{thm:polyhedra} is the ability to handle ideals which can be defined by polyhedra other than shifted cones, e.g.\  integrally closed ideals, symbolic powers, and multiplier ideals.  See, for instance, Theorem~\ref{thm:powerspolyhedra} below.
\end{rmk}

The power in the above theorem comes from the ability to compute face data of polyhedra after fundamental operations. Specifically, for $\bd$ a vector, $P_1$ and $P_2$ polyhedra, four of the fundamental operations are the following operations. 

\

\begin{center}
\begin{tabular}{cc}
    Translating & $P_1 + \bd$ \\
    Scaling & $cP_1$ \\
    Intersection & $P_1 \cap P_2$ \\
    Minkowski sum & $P_1 + P_2$.
\end{tabular}
\end{center}

\

Setting $I_1$ and $I_2$ the ideals associated to $P_1$ and $P_2$, the combinatorial geometric operations on polyhedra of course produce a polyhedron whose associated ideal arises from expected ideal constructions, specifically, the ideal associated to $P_1 + P_2$ is the ideal product $I_1 I_2$.

\begin{rmk}All fundamental operations described trivially preserve the $\sigma^{\vee}$-closed conditions.  

\

One should consider unions as a fundamental operation, but as defined the union of polyhedra may fail to be a polyhedra and taking the convex hull is a lossy process in that working this way would not allow us to distinguish between ideals and their integral closures. As such, the statement of Theorem~\ref{thm:polyhedra} assumes polyhedral input but works to describe the ideal associated to the union so this operation is not lost in our application. 
\end{rmk}

\

Note, in total generality, computing face data of the output of some of these fundamental operations from the face data of its inputs can be computationally difficult, notably NP-hard even in more restrictive settings than our general polyhedra \cite{Tiw08}. However, we quickly review some easy and rather explicit computations in some specialized settings sufficient for our applications.

\begin{lem}\label{lem:facedataop} Fix $P_1$ and $P_2$ $\sigma^{\vee}$-closed polyhedra with face data $\{(g_{\tau},M_{\tau})\}_{\tau \in \mathcal{S}_i}$ for $i = 1,2$, $\bd$ a vector, and $c>0$.
\begin{enumerate}
    \item \label{item:shifted polytope} $P_1+\bd$ is defined by face data $\{(g_\tau,M_\tau+g_\tau(\bd))\}_{\tau\in \mathcal{S}_1}$
    \item \label{item:scaled polytope} $cP_1$ is defined by face data $\{(g_\tau,cM_\tau)\}_{\tau\in \mathcal{S}_1}$
    \item \label{item:intersection of polytopes} $P_1\cap P_2$ is defined by face data $\{(g_{\tau},M_{\tau})\}_{\tau \in \mathcal{S}_1\cup\mathcal{S}_2}$
    \item \label{item:sum of polytopes} $P_1+P_2$ is defined by face data $\{(g_{\tau}, N_\tau)\}_{\tau\in\mathcal{S}_1\cup\mathcal{S}_2}$ for some $N_\tau\in \RR$.  In the two-dimensional case, for each $\tau\in \mathcal{S}_i$, we have that $N_\tau = M_\tau+g_\tau(\ba_\tau)$, where $\ba_\tau$ is a vertex of $P_{3-i}$ such that line through $a_\tau$ parallel to $\tau$ does not intersect the interior of $P_{3-i}$.
\end{enumerate}
\end{lem}

\begin{proof} (\ref{item:shifted polytope}) We have that $\bv\in P_1+\bd$ if and only if $\bv - \bd\in P_1$, which occurs exactly when $g_\tau(\bv-\bd)\geq M_\tau$ for all $\tau\in \mathcal{S}_1$.  Therefore $\bv\in P_1+\bd$ if and only if $g_\tau(\bv)\geq M_\tau+g_\tau(\bd)$ for all $\tau\in \mathcal{S}_1$.

(\ref{item:scaled polytope}) We have that $\bv\in cP_1$ if and only if $(1/c)\bv\in P_1$, which occurs exactly when $g_\tau((1/c)\bv)\geq M_\tau$ for all $\tau\in \mathcal{S}_1$.  Therefore $\bv\in cP_1$ if and only if $g_\tau(\bv)\geq cM_\tau$ for all $\tau\in \mathcal{S}_1$.

(\ref{item:intersection of polytopes}) We have that $\bv\in P_1\cap P_2$ if and only if $\bv$ is in both $P_1$ and $P_2$, i.e.\ if $g_\tau(\bv)\geq M_\tau$ for all $\tau\in \mathcal{S}_1\cup \mathcal{S}_2$.

(\ref{item:sum of polytopes}) The result is well-known, and algorithms for computing the numbers $N_\tau$ can be found in \cite[Ch. 13]{dBCvKO08}.  However, there are some technical adaptations to derive the stated claim. Notably, we are working with $\sigma^{\vee}$-closed regions which are not compact and also any two polyhedra have parallel faces. 

We note the selection criteria for which $\ba_\tau$ is chosen is simply done by linear programming. Thus, for each face $\tau\in \mathcal{S}_1$, there is a face $\tau'$ of $P_1+P_2$ which is defined by $\tau'=\tau + \ba_\tau$, where $\ba_\tau$ is a vertex of $P_2$ which minimizes $g_\tau$, and similarly for every $\tau\in \mathcal{S}_2$, there is a face $\tau'=\tau+\ba_\tau$, where $\ba_\tau$ is a vertex of $P_1$.  A vector $\bv$ is in the half-plane defined by $\tau'$ if and only if $\bv-\ba_\tau$ is in the half-plane defined by $\tau$, which occurs exactly when $g_\tau(\bv-\ba_\tau)\geq M_\tau$, implying $g_\tau(\bv) \geq M_\tau + g_\tau(\ba_\tau)$.  This proves the last statement of (\ref{item:sum of polytopes}).
\end{proof}

\begin{rmk}
The primary consequence of Lemma~\ref{lem:facedataop} is that combined with Theorem~\ref{thm:polyhedra} gives an algorithm for computing face data of many common polyhedra and the differential operators fixing the associated monomial ideals. 
\end{rmk}

\subsection{Symbolic powers and Integral closure of powers} We now turn out attention to sharper results in the case of ideal membership problems for the integral closure of powers \[\overline{I^n}=\{x \in R \mid \text{there exists } n \in \NN \text{ and }a_i \in I^{ni} \text{ for } 1 \leq i \leq n \text{ such that }  x^n+a_1x^{n-1}+\cdots +a_{n-1}x+a_n=0 \}\] and the symbolic power $I^{(n)} = \bigcap\limits_{\fp \in {\rm Ass}(I)} I^nR_{\fp} \cap R$. For monomial ideals $I$, there are important descriptions of these in terms of a polyhedra.

\begin{dff}
Let $R=\CC[S]$ be an affine semigroup ring defined by a cone $\sigma^{\vee}$ and $I \subseteq R$ a proper monomial ideal with primary decomposition $I= Q_1 \cap \cdots Q_r$.  For a prime $P\in {{\rm maxAss}(I)}$ set $Q_{\subseteq P}=IR_P \cap R$ (the intersection of the $Q_i \subseteq P$).
\begin{enumerate}
    \item The {\em{Newton polyhedron}} of $I$ is $\Newt{(I)}={\rm convex\  hull}\langle\Exp{(I)}\rangle$.
    \item  The {\em{symbolic polyhedron}} of $I$ is ${\rm SP}(I)= \bigcap\limits_{\fp \in {\rm maxAss}(I)} \Newt{(Q_{\subseteq P})}$.
\end{enumerate} 
\end{dff}

\begin{thm}\label{thm:powerspolyhedra}
Let $R=\CC[S]$ be an affine semigroup ring defined by a cone $\sigma^{\vee}$ and $I \subseteq R$ a proper monomial ideal.
\begin{enumerate}
    \item \cite[Proposition 2.7]{DFMS19} $x^{\bd} \in \overline{I^n}$ if and only if $\displaystyle\frac{\bd}{n} \in \Newt{(I)}$.
    \item \cite[Theorem 5.4]{CEHH17} If $x^{\bd} \in I^{(n)}$, then $\displaystyle\frac{\bd}{n} \in {\rm SP}{(I)}$.
    \item \cite[Proposition 2.10]{DFMS19} If we assume $S= \mathbb{N}^d$ for some $d$, i.e. $R$ is a polynomial ring and $I$ is a squarefree monomial ideal, then $x^{\bd} \in I^{(n)}$ if and only if $\displaystyle\frac{\bd}{n} \in {\rm SP}{(I)}$.
    \end{enumerate}
\end{thm}

\begin{rmk} Teissier notes in the exercises of \cite[Section 3]{Tei02} that the proof of (1) above also holds for any normal affine semigroup ring. For (2), the proof provided in \cite[Theorem 5.4]{CEHH17} goes through for normal affine semigroup rings. For (3), we believe that the generalization from $S= \NN^d$ to any strongly convex semigroup $S$ is probably true for ideals $I$ which are intersections of monomial primes. While we don't pursue this in full detail here, Example \ref{ex:sqfreegen} illustrates a normal affine semigroup ring and an ideal $I$ which is an intersection of all height 2 primes and it is clear that $n{\rm SP}(I)=\Exp{(I^{(n)})}.$ 
\end{rmk}

\begin{rmk}
Theorem \ref{thm:powerspolyhedra} (3) also holds for a wider class of ideals as noted in \cite[Remark 2.6]{HN21}, namely when $I$ is of linear-power type, specifically for those ideals  $I= \bigcap\limits_{\fp \in {\rm Min}(I)} {\fp}^{\omega_{\fp}}$ for some $\omega_{\fp} \in \NN$ and all minimal primes are monomial primes.
\end{rmk}

\begin{xmp}\label{ex:sqfreegen}
Let $R=\CC[x,y,z,xyz^{-1}]$ and \[I=(xy,xz,yz,x^2yz^{-1},xy^2z^{-1})=(x,y,z) \cap (y,z,xyz^{-1}) \cap (x,y,xyz^{-1}) \cap (x,z,xyz^{-1}).\] (The intersection of the height 2 primes of $R$; the equivalent of a square-free monomial ideal in a polynomial ring.) The facets of $R$ are 
\[\sigma_1=\{(x,y,z) \mid y=0\}, \sigma_2=\{(x,y,z) \mid y+z=0\}, \sigma_3=\{(x,y,z) \mid x+z=0\}, \text{ and } \sigma_4=\{(x,y,z) \mid x=0\},
\]
and the polyhedra defining the primes are 
\begin{align*}
    P_1&=\{(x,y,z) \mid x+y \geq 1, y \geq 0, y+z \geq 0, x+z \geq 0, x \geq 0 \},\\
    P_2&=\{(x,y,z) \mid 2y+z \geq 1, y \geq 0, y+z \geq 0, x+z \geq 0, x \geq 0\},\\
    P_3&=\{(x,y,z) \mid x+y+2z \geq 1, y \geq 0, y+z \geq 0, x+z \geq 0, x \geq 0\},\\
    P_4&=\{(x,y,z) \mid 2x+z \geq 1, y \geq 0, y+z \geq 0, x+z \geq 0, x \geq 0\}.\\
\end{align*}
The symbolic polyhedron \begin{align*}
    {\rm SP}(I)&= P_1 \cap P_2\cap P_3 \cap P_4\\
    &=\{(x,y,z) \mid x \geq 0, y \geq 0, x+z \geq 0, y+z \geq 0,x+y \geq 1, 2y+z \geq 1,x+y+2z \geq 1 \text{ and } 2x+z \geq 1\}\end{align*} is a polyhedron with vertex $(1/2,1/2,0)$.  It is easy to see that $I^{(2)}=(xy,x^2z^2,y^2z^2,x^4y^2z^{-2},x^2y^4,z^{-2})$ is the ideal generated by $2S(P)$ giving credence to the fact that Theorem \ref{thm:powerspolyhedra} (3) also holds for ideals which are intersections of homogeneous primes defined by the faces of a normal affine semigroup ring. For an example of a differential operator fixing $2{\rm SP}(I)=\Exp{(I^{(2)})}$, consider $\bd=(-1,-1,0)$; then $H_\bd=\theta_x\theta_y(\theta_x+\theta_z)(\theta_{y}+\theta_z)$ and \[f=\theta_x\theta_y(\theta_x+\theta_z)(\theta_{y}+\theta_z)(\theta_x+\theta_y,3)!(2\theta_x+\theta_z,3)!(2\theta_y+\theta_z,3)!(\theta_x+\theta_y+2\theta_z,3)!\] and $\delta=x^{-1}y^{-1} f$ fixes $I^{(2)}$.
    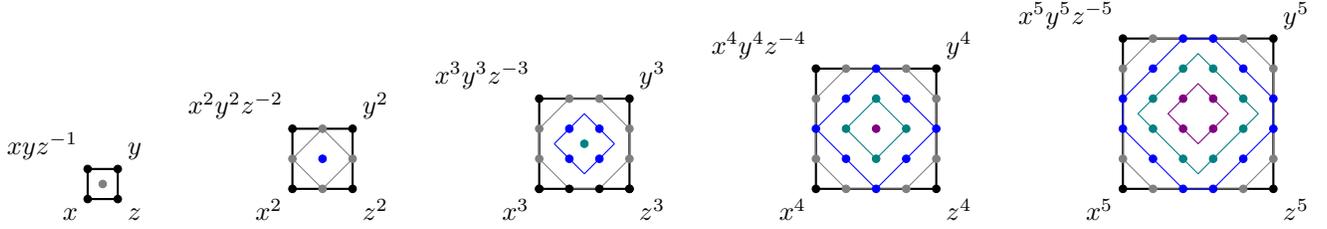
\begin{figure}[h]
  \centering
 \begin{tikzpicture}[scale=0.4]
\draw[black, thick] (0,0) node[below left]{$x$} -- (1,0) node[below right]{$z$};
\draw[black, thick] (0,0) -- (0,1);
\draw[black, thick] (1,1) -- (1,0);
\draw[black, thick] (1,1) node[above right]{$y$} -- (0,1) node[above left]{$xyz^{-1}$};
\foreach \x in {.5}{ \node[draw,circle,inner sep=1pt,gray,fill] at (\x,\x) {};}
\foreach \x in {0,1}{ \node[draw,circle,inner sep=1pt,black,fill] at (\x,1) {};}
\foreach \x in {0,1}{ \node[draw,circle,inner sep=1pt,black,fill] at (\x,0) {};}
\end{tikzpicture}
\hspace{.15cm}
 \begin{tikzpicture}[scale=0.4]
\draw[black, thick] (0,0) node[below left]{$x^2$} -- (2,0) node[below right]{$z^2$};
\draw[black, thick] (0,0) -- (0,2);
\draw[black, thick] (2,2) -- (2,0);
\draw[black, thick] (2,2) node[above right]{$y^2$} -- (0,2) node[above left]{$x^2y^2z^{-2}$};
\draw[gray] (0,1) --(1,2) --(2,1) -- (1,0) --(0,1);
\foreach \x in {0,2}{ \node[draw,circle,inner sep=1pt,gray,fill] at (\x,1) {};}
\foreach \x in {0,2}{ \node[draw,circle,inner sep=1pt,black,fill] at (\x,0) {};}
\foreach \x in {1}{ \node[draw,circle,inner sep=1pt,gray,fill] at (\x,0) {};}
\foreach \x in {0,2}{ \node[draw,circle,inner sep=1pt,black,fill] at (\x,2) {};}
\foreach \x in {1}{ \node[draw,circle,inner sep=1pt,gray,fill] at (\x,2) {};}
\foreach \x in {1}{ \node[draw,circle,inner sep=1pt,blue,fill] at (\x,1) {};}
\end{tikzpicture}
\hspace{.15cm}
 \begin{tikzpicture}[scale=0.4]
\draw[black, thick] (0,0) node[below left]{$x^3$} -- (3,0) node[below right]{$z^3$};
\draw[black, thick] (0,0) -- (0,3);
\draw[black, thick] (3,3) -- (3,0);
\draw[black, thick] (3,3) node[above right]{$y^3$} -- (0,3) node[above left]{$x^3y^3z^{-3}$};
\draw[gray] (0,2) --(1,3);
\draw[gray] (2,3) -- (3,2);
\draw[gray] (3,1) -- (2,0);
\draw[gray] (1,0) -- (0,1);
\draw[gray] (1,0) -- (2,0);
\draw[gray] (0,2) -- (0,1);
\draw[gray] (3,1) -- (3,2);
\draw[gray] (1,3) -- (2,3);
\draw[blue] (1.5,.5) -- (.5,1.5);
\draw[blue] (.5,1.5) -- (1.5,2.5);
\draw[blue] (1.5,2.5) -- (2.5,1.5);
\draw[blue] (2.5,1.5) -- (1.5,0.5);
\foreach \x in {0,1,2,3}{ \node[draw,circle,inner sep=1pt,black,fill] at (\x,0) {};}
\foreach \x in {0,3}{ \node[draw,circle,inner sep=1pt,black,fill] at (\x,0) {};}
\foreach \x in {1,2}{ \node[draw,circle,inner sep=1pt,gray,fill] at (\x,3) {};}
\foreach \x in {0,3}{ \node[draw,circle,inner sep=1pt,black,fill] at (\x,3) {};}
\foreach \x in {1,2}{ \node[draw,circle,inner sep=1pt,blue,fill] at (\x,1) {};}
\foreach \x in {0,3}{ \node[draw,circle,inner sep=1pt,gray,fill] at (\x,1) {};}
\foreach \x in {1,2}{ \node[draw,circle,inner sep=1pt,blue,fill] at (\x,2) {};}
\foreach \x in {0,3}{ \node[draw,circle,inner sep=1pt,gray,fill] at (\x,2) {};}
\foreach \x in {1.5}{ \node[draw,circle,inner sep=1pt,teal,fill] at (\x,\x) {};}
\end{tikzpicture}
\hspace{.15cm}
 \begin{tikzpicture}[scale=0.4]
\draw[black, thick] (0,0) node[below left]{$x^4$} -- (4,0) node[below right]{$z^4$};
\draw[black, thick] (0,0) -- (0,4);
\draw[black, thick] (4,4) -- (4,0);
\draw[black, thick] (4,4) node[above right]{$y^4$} -- (0,4) node[above left]{$x^4y^4z^{-4}$};
\draw[gray] (0,3) --(1,4) -- (3,4) -- (4,3) -- (4,1)  -- (3,0) --(1,0) -- (0,1) -- (0,3);
\draw[blue] (2,0) -- (0,2);
\draw[blue] (0,2) -- (2,4);
\draw[blue] (2,4) -- (4,2);
\draw[blue] (4,2) -- (2,0);
\draw[teal] (2,1) --(1,2)--(2,3) -- (3,2) -- (2,1);
\foreach \x in {0,4}{ \node[draw,circle,inner sep=1pt,black,fill] at (\x,0) {};}
\foreach \x in {1,3}{ \node[draw,circle,inner sep=1pt,gray,fill] at (\x,0) {};}
\foreach \x in {2}{ \node[draw,circle,inner sep=1pt,blue,fill] at (\x,0) {};}
\foreach \x in {0,4}{ \node[draw,circle,inner sep=1pt,gray,fill] at (\x,3) {};}
\foreach \x in {1,3}{ \node[draw,circle,inner sep=1pt,blue,fill] at (\x,3) {};}
\foreach \x in {2}{ \node[draw,circle,inner sep=1pt,teal,fill] at (\x,3) {};}
\foreach \x in {0,4}{ \node[draw,circle,inner sep=1pt,gray,fill] at (\x,1) {};}
\foreach \x in {1,3}{ \node[draw,circle,inner sep=1pt,blue,fill] at (\x,1) {};}
\foreach \x in {2}{ \node[draw,circle,inner sep=1pt,teal,fill] at (\x,1) {};}
\foreach \x in {0,4}{ \node[draw,circle,inner sep=1pt,blue,fill] at (\x,2) {};}
\foreach \x in {1,3}{ \node[draw,circle,inner sep=1pt,teal,fill] at (\x,2) {};}
\foreach \x in {2}{ \node[draw,circle,inner sep=1pt,violet,fill] at (\x,2) {};}
\foreach \x in {0,4}{ \node[draw,circle,inner sep=1pt,black,fill] at (\x,4) {};}
\foreach \x in {1,3}{ \node[draw,circle,inner sep=1pt,gray,fill] at (\x,4) {};}
\foreach \x in {2}{ \node[draw,circle,inner sep=1pt,blue,fill] at (\x,4) {};}
\end{tikzpicture}
\hspace{.15cm}
 \begin{tikzpicture}[scale=0.4]
\draw[black, thick] (0,0) node[below left]{$x^5$} -- (5,0) node[below right]{$z^5$};
\draw[black, thick] (0,0) -- (0,5);
\draw[black, thick] (5,5)-- (5,0);
\draw[black, thick] (5,5)node[above right]{$y^5$} -- (0,5) node[above left]{$x^5y^5z^{-5}$};
\draw[gray] (0,4) --(1,5) -- (4,5) -- (5,4) -- (5,1)  -- (4,0) --(1,0) -- (0,1) -- (0,4);
\draw[blue] (2,0) -- (0,2) -- (0,3) -- (2,5) -- (3,5) -- (5,3) -- (5,2) -- (3,0) -- (2,0);
\draw[teal] (2.5,.5) -- (.5,2.5) -- (2.5, 4.5) -- (4.5,2.5) -- (2.5, .5);
\draw[violet] (2.5,1.5) --(1.5,2.5) -- (2.5,3.5) -- (3.5,2.5) -- (2.5,1.5);
\foreach \x in {0,5}{ \node[draw,circle,inner sep=1pt,black,fill] at (\x,0) {};}
\foreach \x in {1,4}{ \node[draw,circle,inner sep=1pt,gray,fill] at (\x,0) {};}
\foreach \x in {2,3}{ \node[draw,circle,inner sep=1pt,blue,fill] at (\x,0) {};}
\foreach \x in {0,5}{ \node[draw,circle,inner sep=1pt,black,fill] at (\x,5) {};}
\foreach \x in {1,4}{ \node[draw,circle,inner sep=1pt,gray,fill] at (\x,5) {};}
\foreach \x in {2,3}{ \node[draw,circle,inner sep=1pt,blue,fill] at (\x,5) {};}
\foreach \x in {0,5}{ \node[draw,circle,inner sep=1pt,gray,fill] at (\x,4) {};}
\foreach \x in {0,5}{ \node[draw,circle,inner sep=1pt,gray,fill] at (\x,1) {};}
\foreach \x in {1,4}{ \node[draw,circle,inner sep=1pt,blue,fill] at (\x,4) {};}
\foreach \x in {1,4}{ \node[draw,circle,inner sep=1pt,blue,fill] at (\x,1) {};}
\foreach \x in {2,3}{ \node[draw,circle,inner sep=1pt,teal,fill] at (\x,4) {};}
\foreach \x in {2,3}{ \node[draw,circle,inner sep=1pt,teal,fill] at (\x,1) {};}
\foreach \x in {0,5}{ \node[draw,circle,inner sep=1pt,blue,fill] at (\x,2) {};}
\foreach \x in {0,5}{ \node[draw,circle,inner sep=1pt,blue,fill] at (\x,3) {};}
\foreach \x in {1,4}{ \node[draw,circle,inner sep=1pt,teal,fill] at (\x,2) {};}
\foreach \x in {1,4}{ \node[draw,circle,inner sep=1pt,teal,fill] at (\x,3) {};}
\foreach \x in {2,3}{ \node[draw,circle,inner sep=1pt,violet,fill] at (\x,2) {};}
\foreach \x in {2,3}{ \node[draw,circle,inner sep=1pt,violet,fill] at (\x,3) {};}
\end{tikzpicture}
\captionsetup{margin=0in,width=5.1in,font=small,justification=centering}
\caption{Cross-sections of multiples of the symbolic polyhedra of $I$}
\label{fig:SP}
 \end{figure}
    
    In Figure \ref{fig:SP}, we illustrate cross-sections of the lattice points of the cone defining $R$.  The lattice points in grey represent the monomials which are in $I \setminus I^{(2)}$; those in blue represent the monomials which are in $I^{(2)} \setminus I^{(3)}$; those in teal represent the monomials which are in $I^{(3)} \setminus I^{(4)}$; and those in violet represent the the monomials which are in  $I^{(4)}$.  We have included the vertices of the symbolic polyhedra ${\rm SP}(I)$ and $3{\rm SP}(I)$ in gray and teal respectively although they do not correspond to monomials.
\end{xmp}

\begin{xmp}
Let $R=\CC[x,y,z]$ and $I=(x^2z,xy^2,yz^2,xyz)=(x,z^2) \cap (x^2,y) \cap (y^2,z)$.  The minimal (and maximal) primes are $P_1=(x,z),P_2=(x,y)$ and $P_3=(y,z)$.  Set $Q_1=(x,z^2), Q_2=(x^2,y)$ and $Q_3=(y^2,z)$. Then $Q_{\subseteq P_i}=Q_i$ and the faces of the Newton Polyhedra of $Q_i$ are the following: 
\[P_1=\{(x,y,z) \mid 2x+z \geq 2, x \geq 0, y \geq 0 \text{ and } z \geq 0\},\] \[ P_2=\{(x,y,z) \mid 2y+x \geq 2, x \geq 0, y \geq 0 \text{ and } z \geq 0\},\] \[ P_3=\{(x,y,z) \mid 2y+z \geq 2, x \geq 0, y \geq 0 \text{ and } z \geq 0\}.
\]
Thus, \[{\rm SP}(I)= P_1 \cap P_2 \cap P_3
\] which is a polyhedral cone with vertex $(2/3,2/3,2/3)$ with facets \[\tau_1=\{(x,y,z) \mid 2x+z = 2\}, \tau_2=\{(x,y,z) \mid 2y+x = 2\}, \text{ and } \tau_3=\{(x,y,z) \mid 2z+y = 2\}.\] In \cite[Example 4.6]{CEHH17}, the authors note that $x^2y^2z^2 \in I^{(3)}$ and Theorem \ref{thm:powerspolyhedra} (2) in fact verifies that the vertex of the cone above lies in ${\rm SP}(I)$.  In fact, it is also easy to verity that $x^{2n}y^{2n}z^{2n} \in P^{(3n)}$. 

Applying Theorem \ref{thm:polyhedra}, we will exhibit a differential operator which fixes the ideal generated by ${\rm SP}(I)$.  Note that the facet data for ${\rm SP}(I)$ is $\{(2x+z,2), (2y+x,2), (2z+y,2)\}$.  Set $\bd=(-1,-1,-1)$ and note that $H_{\bd}=\theta_x\theta_y\theta_z$ and $f=(2\theta_x+\theta_z, 4)!(2\theta_y+\theta_x, 4)!(2\theta_z+\theta_y, 4)!$.  Since $H_{\bd}f\in (H_{\bd}) \cap (f)$, then $\delta=x^\bd H_{\bd}f$ fixes the ideal generated by ${\rm SP}(I)$. 
\end{xmp}

\begin{rmk}
In the above example, we can replace 2 by any natural number $n$ and the facet data for ${\rm SP}(I_n)$ for $I_n=(x^nz,xy^n,yz^n,xyz)=(x,z^n) \cap (x^n,y) \cap (y^n,z)$ will instead be $\{(nx+z,n),(ny+x,n),(nz+y,n)$ and ${\rm SP}(I)$ will be a cone with vertex $(\frac{n}{n+1},\frac{n}{n+1},\frac{n}{n+1})\}$.  It is interesting to note that as $n$ gets large the vertices of these cones are approaching $(1,1,1)$.  Of course, the ideals $I_n$ are just nilpotent thickenings of $I_1=(x,y) \cap (x,z) \cap (y,z)=(xy,xz,yz)$.  As $n$ gets large,  the symbolic polyhedron of these thickenings is approaching the cone of the ideal $J=(xyz)=(x) \cap (y) \cap (z)$, the intersection of height 1 primes.

Similarly, when $a<b$ and $I=(x^az^b,x^by^a,y^bz^a,x^ay^az^a)=(x^b,z^a) \cap (x^a,y^b) \cap (y^a,z^b)$, the facet data for  ${\rm SP}(I)$ will instead be $\{(ax+bz,ab),(ay+bx,ab),(az+by,ab)$ and ${\rm SP}(I)$ will be a cone with vertex $(\frac{ab}{a+b},\frac{ab}{a+b},\frac{ab}{a+b})\}$.  Again these ideals are thickenings of $J=(x^ay^a,x^az^a,y^a,z^a)$ and their symbolic polyhedra again are converging toward the cone of the ideal $(x^ay^az^a)$.
\end{rmk}

\begin{xmp}
The vertex of the cone need not be a multiple of $(1,1,1)$. Consider \[I=(x^ay^b,x^az^c,y^bz^c)=(x^a,z^c) \cap (x^a,y^b) \cap (y^b,z^c).\]  The polyhedra defining the primary components are
\[P_1=\{(x,y,z) \mid cx+az \geq ac, x \geq 0, y \geq 0, \text{ and } z \geq 0\},
\]
\[P_2=\{(x,y,z) \mid bx+ay \geq ab, x \geq 0, y \geq 0, \text{ and } z \geq 0\},
\]
\[P_3=\{(x,y,z) \mid cy+bz \geq bc, x \geq 0, y \geq 0, \text{ and } z \geq 0\}.
\]
Thus, \[{\rm SP}(I)= P_1 \cap P_2 \cap P_3
\] which is a polyhedral cone with vertex $(\frac{a}{2},\frac{b}{2},\frac{c}{2})$ with facets \[\tau_1=\{(x,y,z) \mid cx+az = ac\}, \tau_2=\{(x,y,z) \mid bx+ay =ab\}, \text{ and } \tau_3=\{(x,y,z) \mid cy+bz=bc\}.\]
It is easy to check that $x^ay^bz^c \in I^{(2)} \setminus I^2$ and $\frac{1}{2}(a,b,c) \in {\rm SP}(I)$ as in Theorem \ref{thm:powerspolyhedra}(b).

Applying Theorem \ref{thm:polyhedra}, we will exhibit a differential operator which fixes the ideal generated by $2{\rm SP}(I)$.  Note that this ideal contains $I^{(2)}$ but is not $I^{(2)}$ since $x^{2a-1}y^bz$ is one example of an element in $2{\rm SP}(I)\setminus \Exp{(I^{(2)})}$.
Note that the facet data for $2{\rm SP}(I)$ is $\{(cx+az,2ac), (bx+ay,2ab), (cy+bz,2bc)\}$.  Set $\bd=(-1,-1,-1)$ and note that $H_{\bd}=\theta_x\theta_y\theta_z$ and $f=(c\theta_x+a\theta_z, 2ac+a+c-1)!(b\theta_x+a\theta_y, 2ab+a+b-1)!(c\theta_y+b\theta_z, 2bc+b+c-1)!$.  Since $H_{\bd}f \in (H_{\bd}) \cap (f)$, then $\delta=x^\bd H_{\bd}f$ fixes the ideal generated by $2{\rm SP}(I)$.
\end{xmp}

\subsection{Multiplier ideals and jumping numbers} We now turn to sharper applications of Theorem~\ref{thm:polyhedra} towards multiplier ideals and their variants. In the general framework, if $X$ is a normal variety, $\Delta$ a $\QQ$-divisor so that $K_X + \Delta$ is $\QQ$-Cartier where $K_X$ is a canonical divisor, $\mathcal{I}$ an ideal sheaf on $X$ and $c \geq 0$ a real number, the multiplier ideals $J(X,\Delta,\mathcal{I}^c)$ are defined in terms of a log resolution of this data. Famously, on toric varieties, so $X = \Spec \CC[S]$, there is a canonical torus invariant choice for $K_X$. Moreover, choosing $\Delta$ torus invariant then there is a multiple $r$ and an exponent $\bu$ so that $\textrm{div }x^{\bu} = r(K_X + \Delta)$, whence to each $\Delta$ one can associate a vector $\bw = \bu/r$. The exponent associated to $\Delta = 0$ is $\bw = (1,1,\ldots,1)$. Finally, for $\mathcal{I}$ defined by a monomial ideal $I$, each $J(X,\Delta, \mathcal{I}^c)$ is a monomial ideal and we have the following explicit description of the multiplier ideals \cite{Bli04,How01}. 

\begin{thm}[Blickle-Howald] The monomial ideal $$J(X,\Delta, \mathcal{I}^c) = \langle x^\ba \in R \colon \ba + \bw \in c \cdot \int \Newt I\rangle$$ where $\int  \Newt I$ denotes the relative interior.
\end{thm}

From here on, we'll use the notation $J(X,\bw,I^c)$ instead of $J(X,\Delta, \mathcal{I}^c)$ for this ideal or even just $J(\bw,c)$ when $X$ and $I$ are clear. As is well-known, as the value of $c$ gets larger, the multiplier ideals $J(X,\bw,I^c)$ form a decreasing family of ideals and values where $J(X,\bw,I^{c - \epsilon}) \neq J(X,\bw,I^{c})$ for all $\epsilon > 0$ are called {\bf jumping numbers} for the triple $(X,\Delta,\mathfrak{a})$. The smallest jumping number is the log canonical threshold which plays a most critical role in the study of singularities. Detecting when a value $c$ is the log canonical threshold therefore comes down to determining if the monomial ideal $J(X,\bw,I^{c})$ is proper or not. Our applications will characterize the jumping numbers $c$ of a monomial ideal $I$ in terms of the differential operators that fix $I$. The following simple example demonstrates the idea. 

\begin{xmp}
Let $R=\CC[x,y]$ and $I=(x^2,y^3)$.  Let $\bw=(1,1)$.  Define 
\[J(\bw,c)=\{x^\bv \mid \bv+\bw \in c \cdot\int{(\Newt(I))}\}.\] Solving for $t$ and $u$, we see that $t=\frac12$ and $u=\frac13$, making $c<\frac56$.

\begin{figure}[htp]
  \centering
\begin{tikzpicture}[scale=0.6]
\draw[black, thick] (0,0) -- (4,0);
\draw[black, thick] (0,0)--(0,4);
\draw[black] (2,0)--(0,3);
\draw[teal] (1.67,0) node[below]{(2c,0)} --(0,2.5) node[left]{(0,3c)};
\draw[red] (0.67,-1) node[below]{(2c-1,-1)} --(-1,1.5) node[left]{(1,3c-1)};
\draw[red] (0.67,-1) -- (4,-1);
\draw[red] (-1,1.5) -- (-1,4);
\draw[teal] (1.67,0) -- (4,0);
\draw[teal] (0,2.5)--(0,4);
\foreach \x in {0,1,2,3,4}{ \node[draw,circle,inner sep=1pt,black,fill] at (\x,0) {};}
\foreach \x in {0,1,2,3,4}{ \node[draw,circle,inner sep=1pt,black,fill] at (\x,1) {};}
\foreach \x in {0,1,2,3,4}{ \node[draw,circle,inner sep=1pt,black,fill] at (\x,2) {};}
\foreach \x in {0,1,2,3,4}{ \node[draw,circle,inner sep=1pt,black,fill] at (\x,3) {};}
\foreach \x in {0,1,2,3,4}{ \node[draw,circle,inner sep=1pt,black,fill] at (\x,4) {};}

\end{tikzpicture}
\captionsetup{margin=0in,width=2.1in,font=small,justification=centering}
\caption{Log canonical threshold}
\end{figure}
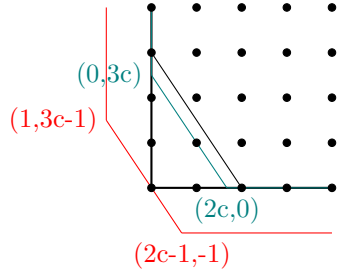 
We want to visualize this threshold $c$ through differential operators. 
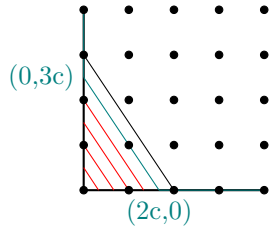
\begin{figure}[h]
  \centering
\begin{tikzpicture}[scale=0.6]
\draw[black, thick] (0,0) -- (4,0);
\draw[black, thick] (0,0)--(0,4);
\draw[black] (2,0)--(0,3);
\draw[teal] (1.67,0) node[below]{(2c,0)} --(0,2.5) node[left]{(0,3c)};
\draw[red] (1.33,0) --(0,2);
\draw[red] (1,0) --(0,1.5);
\draw[red] (.67,0) --(0,1);
\draw[red] (.33,0) --(0,.5);
\draw[teal] (1.67,0) -- (4,0);
\draw[teal] (0,2.5)--(0,4);
\foreach \x in {0,1,2,3,4}{ \node[draw,circle,inner sep=1pt,black,fill] at (\x,0) {};}
\foreach \x in {0,1,2,3,4}{ \node[draw,circle,inner sep=1pt,black,fill] at (\x,1) {};}
\foreach \x in {0,1,2,3,4}{ \node[draw,circle,inner sep=1pt,black,fill] at (\x,2) {};}
\foreach \x in {0,1,2,3,4}{ \node[draw,circle,inner sep=1pt,black,fill] at (\x,3) {};}
\foreach \x in {0,1,2,3,4}{ \node[draw,circle,inner sep=1pt,black,fill] at (\x,4) {};}

\end{tikzpicture}
\captionsetup{margin=0in,width=2.3in,font=small,justification=centering}
\caption{LCT/Differential operator}
\label{fig:lct}
\end{figure}

To determine which values $c$ satisfy $\b0 \in \int{(\Exp{J(\bw,c)})}$, note that \[(1,1)=t(2,0)+u(0,3) \] for $t,u \geq 0$, $t+u>c$.  See Figure \ref{fig:lct}. Let $f_c(\theta)=H_{(-1,-1)}(\theta)\prod\limits_{i=0}^{\lfloor 6c \rfloor}(3 \theta_x-2 \theta_y -i)$.  $\delta=x^{-1}y^{-1}f_c(\theta)$ fixes $J_c=R$ if and only if $(1,1) \notin V(f_c)$.  Note that $f_c(1,1)=5\ldots(5-\lfloor 6c\rfloor) \neq 0$ if $c<5/6$.
\end{xmp}

Thus a direct application of Theorem~\ref{thm:polyhedra} for the shifted Newton polygon gives the following explicit manner of computing multiplier ideals and their jumping numbers from differential operators. 
 
\begin{cor}\label{cor:lct} Let $\bw\in S$, $c\geq 0$ a real number, and $J(X,\bw,I^c) =\langle x^\ba \mid \ba + \bw \in c \cdot \inter \Newt I\rangle$. Set $(g_\tau,M_\tau)_{\tau \in \mathcal{F}}$ a face data for $\Newt I$.  
\begin{enumerate}
    \item The differential operator $\delta_c = x^{-\bw} H_{-\bw}(\theta)\prod_{\tau\in \mathcal{F}}(g_\tau(\theta), \lfloor cM_\tau\rfloor)!$ fixes $J_c$.
    \item The jumping numbers of $I$ are exactly $\left\{\min_{\tau\in\mathcal{F}}\left\{\frac{1}{M_\tau}g_\tau(\ba+\bw)\right\}\mid \ba\in S\right\}$.
    \item The log canonical threshold of $I$ is  $\min_{\tau\in\mathcal{F}}\left\{\frac{1}{M_\tau}g_\tau(\bw)\right\}$.
\end{enumerate}
\end{cor}
\begin{proof}
The first claim holds directly by Theorem~\ref{thm:polyhedra}. A monomial $x^\ba$ is in $J_c$ if and only if $\delta_c(x^{\ba+\bw})\neq 0$, which occurs if and only if $f_c(\ba+\bw)\neq 0$.  This happens exactly when for all $\tau\in\mathcal{F}$, $g_\tau(\ba+\bw)> cM_\tau$, i.e.\ when $c < \frac{1}{M_\tau}g_\tau(\ba+\bw)$.  In particular, this means that $\min\left\{\frac{1}{M_\tau}g_\tau(\ba+\bw)\mid \tau\in\mathcal{F}\right\}$ is a jumping number for $I$.
\end{proof}

\begin{xmp}
Let $R=\CC[x,xy,xy^2]$ and $I=(x^4,x^2y,x^3y^6)$.  For $\bw \in S$, set $$J_c=\{x^{\bv} \in R \mid \bv+\bw \in c\int(\Newt I)\}.$$  Let us assume that $\bw=(1,1)$.    The linear forms associated to the two faces $\tau_1$ and $\tau_2$ that bound the convex hull of $I$ are $g_{\tau_1}:=x+2y \text{ and } g_{\tau_2}:=5x-y$ and it is easy to see that $M_{\tau_1}=4$ and $M_{\tau_2}=9$.  We define \[\delta_c=x^{-1}y^{-1}H_{(-1,-1)}(\theta)(g_{\tau_1}(\theta),\lfloor 4c \rfloor)!(g_{\tau_2}(\theta),\lfloor 9c \rfloor)!.\]
By Corollary~\ref{cor:lct}(2), the jumping numbers of $I$ with respect to the divisor $0$ are \[\left\{\text{min}\left\{\frac{g_{\tau_1}(\ba+(1,1))}{4},\frac{g_{\tau_2}(\ba+(1,1))}{9}\right\}\mid \ba \in S\right\}.\]

The table in Figure \ref{fig:jnosrnc2} illustrates the jumping numbers in ascending order in blue. Note that the jumping numbers can be determined by $\tau_1$ or $\tau_2$.  The illustration in Figure \ref{fig:jnosrnc2} illustrates the boundaries of $c\int{(\Newt{I})}$ for the various jumping numbers in the table.

\begin{figure}[h]
  \centering
\begin{tikzpicture}
\tikzstyle{every column}=[draw]
\matrix [draw=black,column sep=.1cm]
{
\node {$\ba$}; & \node{$\displaystyle\frac{g_{\tau_1}(\ba+(1,1))}{4}$}; & \node {$\displaystyle\frac{g_{\tau_2}(\ba+(1,1))}{9}$}; \\
\node {(0,0)}; & \node{3/4}; & \node {\textcolor{blue}{4/9}}; \\
\node {(1,2)}; & \node{2}; & \node {\textcolor{blue}{7/9}}; \\
\node {(1,1)}; & \node{3/2}; & \node {\textcolor{blue}{8/9}}; \\
\node {(1,0)}; & \node{\textcolor{blue}{1}}; & \node {\textcolor{blue}{1}}; \\
\node {(2,4)}; & \node{13/4}; & \node {\textcolor{blue}{10/9}}; \\
\node {(2,3)}; & \node{11/4}; & \node {\textcolor{blue}{11/9}}; \\
\node {(2,0)}; & \node{\textcolor{blue}{5/4}}; & \node {14/9}; \\
\node {(2,2)}; & \node{9/4}; & \node {\textcolor{blue}{12/9}}; \\
\node {(3,6)}; & \node{15/4}; & \node {\textcolor{blue}{13/9}}; \\
\node {(3,5)}; & \node{13/4}; & \node {\textcolor{blue}{14/9}}; \\
};
\end{tikzpicture}
\hspace{1cm}
\begin{tikzpicture}[scale=0.7]
\filldraw[fill=gray!30, 
draw=white] (4,0) -- (6,0) -- (6,9) -- (4.5,9) -- (3,6)--(2,1)--(4,0) -- cycle;
\draw[gray] (0,0) -- (0,9);
\draw[black, thick] (0,0) -- (6,0);
\draw[black, thick] (0,0)--(4.5,9);
\draw[teal] (1.33,2.66) -- (.88,.44);
\draw[teal] (.88,.44)-- (1.76,0);
\draw[violet] (2.33,4.66) -- (1.55,.77);
\draw[violet](1.55,.77)-- (3.11,0);
\draw[red] (2.66,5.33) -- (1.77,.88);
\draw[red](1.77,.88)-- (3.55,0);
\draw[magenta] (3,6)--(2,1);
\draw[magenta](2,1)--(4,0);
\draw[orange] (3.33,6.66)--(2.22,1.11);
\draw[orange](2.22,1.11)--(4.44,0);
\draw[green] (3.66,7.33)--(2.44,1.22);
\draw[green](2.44,1.22)--(4.88,0);
\draw[blue!80] (3.75,7.50)--(2.5,1.25);
\draw[blue!80](2.5,1.25)--(5,0);
\draw[blue!50] (4,8)--(2.66,1.33);
\draw[blue!50](2.66,1.33)--(5.33,0);
\draw[gray] (4.33,8.66)--(2.88,1.44);
\draw[gray](2.88,1.44)--(5.77,0);
\foreach \x in {0,1,2,3,4,5,6}{ \node[draw,circle,inner sep=1pt,black,fill] at (\x,0) {};}
\foreach \x in {1,2,3,4,5,6}{ \node[draw,circle,inner sep=1pt,black,fill] at (\x,1) {};}
\foreach \x in {1,2,3,4,5,6}{ \node[draw,circle,inner sep=1pt,black,fill] at (\x,2) {};}
\foreach \x in {2,3,4,5,6}{ \node[draw,circle,inner sep=1pt,black,fill] at (\x,3) {};}
\foreach \x in {2,3,4,5,6}{ \node[draw,circle,inner sep=1pt,black,fill] at (\x,4) {};}
\foreach \x in {3,4,5,6}{ \node[draw,circle,inner sep=1pt,black,fill] at (\x,5) {};}
\foreach \x in {3,4,5,6}{ \node[draw,circle,inner sep=1pt,black,fill] at (\x,6) {};}
\foreach \x in {4,5,6}{ \node[draw,circle,inner sep=1pt,black,fill] at (\x,7) {};}
\foreach \x in {4,5,6}{ \node[draw,circle,inner sep=1pt,black,fill] at (\x,8) {};}
\foreach \x in {5,6}{ \node[draw,circle,inner sep=1pt,black,fill] at (\x,9) {};}
\end{tikzpicture}
\captionsetup{margin=0in,width=2.1in,font=small,justification=centering}
\caption{Jumping numbers}
\label{fig:jnosrnc2}
\end{figure}
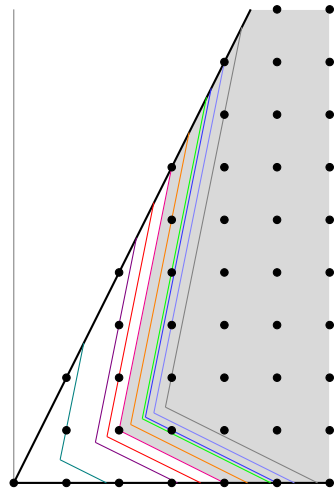
Looking at Figure \ref{fig:jnosrnc2}, we see that the ideals $J_c$ for the jumping numbers given are:
\[
J_c=\begin{cases}
R &\text{ if } c < 4/9\\
\fm &\text{ if } 4/9 \leq  c < 7/9\\
(x,xy,x^2y^4) &\text{ if } 7/9 \leq  c < 8/9\\
(x,x^2y^3,x^2y^4) &\text{ if } 8/9 \leq c < 1\\
\fm^2 &\text{ if } 1 \leq c < 10/9\\
(x^2,x^2y,x^2y^2,x^2y^3,x^3y^6) &\text{ if } 10/9 \leq c <11/9\\
(x^2,x^2y,x^2y^2,x^3y^5,x^3y^6) &\text{ if } 11/9 \leq c < 5/4\\
(x^3,x^2y,x^2y^2,x^3y^5,x^3y^6) &\text{ if } 5/4 \leq c<12/9\\
(x^3,x^2y,x^3y^4,x^3y^5,x^3y^6) &\text{ if } 12/9 \leq c< 13/9\\
(x^3,x^3y,x^3y^2,x^3y^3,x^3y^4,x^3y^5,x^4y^8) &\text{ if } 13/9 \leq c< 14/9\\
\end{cases}
\]

\end{xmp}

Often in the study of multiplier ideals, natural useful generalizations emerge and our methods apply in much more generally to these settings. Indeed, many variants of multiplier ideals, such a mixed multiplier ideals \cite[Gen. 9.2.8]{Laz04}, Fujino's non-LC ideals, and many of the intermediate adjoint ideals, as well as mixed versions, introduced in \cite[Sec. 4,5]{HSZ14} all enjoy variants of the Blickle-Howald theorem. As such, we articulate this more general class as follows.

\begin{dff}
Let $I_1,\ldots, I_n$ be monomial ideals, $\bc\in \RR_{\geq 0}^n$, $\bw\in\ZZ^d$. A {\it Blickle-Howald ideal} is an ideal of the form 
\begin{eqnarray*}J_{I_1,\ldots,I_n}(\bw,\bc) & := & \left\langle x^\ba \mid \ba + \bw 
\in \sum_{j=1}^n  c_j\Newt I_j \right\rangle \textrm{ or }\\
J^{\circ}_{I_1,\ldots,I_n}(\bw,\bc) & := & \left\langle x^\ba \mid \ba + \bw 
\in \sum_{j={1}}^n c_j \inter\Newt I_j\right\rangle.
\end{eqnarray*}
\end{dff} 

\begin{xmp} A few natural examples are the mixed multiplier ideals already mentioned $J(X,\Delta,I_1^{a_1}\cdots I_n^{a_n})$ which by definition is the ideal $J_{I_1,\ldots,I_n}^\circ(\bw, \ba)$ where $\bw$ corresponds to $\Delta$ and $\ba = (a_1,\ldots,a_n)$. 
\end{xmp}

\begin{xmp} For fixed ideals $I_1$ and $I_2$, the ideal $J_{I_1,I_2}(\bw,(c_1,c_2)) = \left\langle x^\ba \mid \ba + \bw 
\in c_1\Newt I_1 + c_2 \Newt I_2 \right\rangle$ is utilized in \cite[Prop. 5.3]{HSZ14}.
\end{xmp}

We refer to BH ideals for Blickle-Howald ideals and when the ideals $I_1,\ldots,I_n$ and the vector $\bw$ they can be suppressed in the notation. As the input data for BH-ideals is more complicated than that for multiplier ideals, instead of having simple jumping numbers, BH ideals enjoy various constancy regions similar to the mixed multiplier ideals, see \cite{LM11} and \cite{ACAMDC18,ACAMDCGA20} for more in that latter case. We summarize the relevant terminology here however the basic results on BH-ideals follow immediately from the methods of \cite{ACAMDC18}. 

\begin{rmk} It is immediate to verify the following from the definition. Set $I_1,\ldots, I_n$ be monomial ideals, $\bc\in \RR_{\geq 0}^n$, $\bw\in\ZZ^d$. For any $\bc' \in \bc + \RR_{\geq 0}^n$, $J^\circ(\bw,\bc) \supseteq J^{\circ}(\bw,\bc')$ and $J^\circ(\bw,\bc) = J^\circ(\bw,\bc')$ for $\bc' \in (\bc + \RR_{\geq 0}^n) \cap B_{\epsilon}(\bc)$ for $0 < \epsilon \ll 1$.
\end{rmk}

Fix $I_1,\ldots,I_n$ and $\bw$. Again, adapting terminology from the mixed multiplier ideal setting, we say for $\bc$ the locus $C_\bc := \{ \bc' \in \RR_{\geq 0}^n \colon J^\circ(\bw,\bc') = J^\circ(\bw,\bc)\}$ is its {\it consistency region}. A {\it jumping point} is a value $\bc$ so that $J^\circ(\bw,\bc') \supsetneq J^\circ(\bw,\bc)$ for all $\bc' \in \{ \bc - \RR_{\geq 0}^n\} \cap B_\epsilon(\bc)$ and $0 < \epsilon \ll 1$. A {\it jumping wall} is the boundary of $C_\bc$. 

\

We are interested in the jumping wall of consistency region where $J^\circ(\bw,\bc) = R$ which in the mixed multiplier ideal is the {\it log-canonical wall}. Clearly, the methods combining Theorem~\ref{thm:polyhedra} and Lemma~\ref{lem:facedataop} combine to extend Corollary~\ref{cor:lct} to many more BH-ideals. We characterize this region in terms of differential operators in the following special case which is an immediate consequence.

\begin{xmp}
Let $R=\CC[x,xy,xy^2]$, $I_1=(x^8y,x^3y^3)$, and $I_2=(x^7y^2,x^4y^6)$.  We have that $\Newt I_1$ is defined by the face data $\{(y,1),(2x+5y,21),(2x-y,3)\}$, and $\Newt I_2$ is defined by the face data $\{(y,2),(4x+3y,34),(2x-y,2)\}$.  For any $\bc=(c_1,c_2)\in\RR_{>0}^2$, we have that $c_1\Newt I_1$ has face data $\{(y,c_1),(2x+5y,21c_1),(2x-y,3c_1)\}$, and $\Newt I_2$ is defined by the face data $\{(y,2c_2),(4x+3y,34c_2),(2x-y,2c_2)\}$. In the first graph in Figure \ref{fig:NewtMin}, we illustrate the Newton Polyhedra of $I_1$ and $I_2$ in gray and blue respectively.  
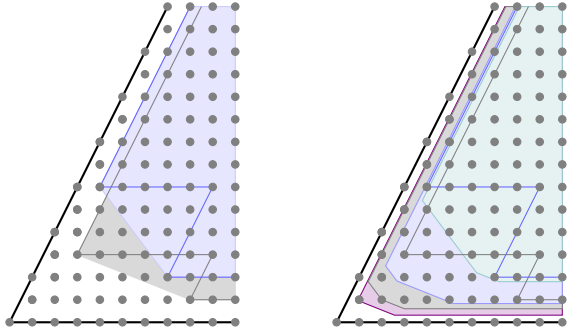
\begin{figure}[h]
  \centering
 \begin{tikzpicture}[scale=0.3]
\filldraw[fill=gray!30,draw=gray!30] (8,1) -- (10,1) -- (10,14) --(8.5,14) -- (3,3) --(8,1) -- cycle; 
\filldraw[fill=blue!10,draw=blue!20] (7,2) -- (10,2) -- (10,14) --(8,14) -- (4,6) --(7,2) -- cycle; 
\draw[black, thick] (0,0) -- (10.1 ,0);
\draw[black, thick] (0,0) -- (7,14);
\draw[gray] (8,1) -- (10,1);
\draw[gray] (3,3) -- (8.5,14);
\draw[gray] (8,1) -- (9,3);
\draw[gray] (3,3) -- (9,3);
\draw[blue!60] (7,2) -- (10,2);
\draw[blue!60] (4,6) -- (8,14);
\draw[blue!60] (7,2) -- (9,6);
\draw[blue!60](4,6) -- (9,6);
\foreach \x in {1,2,3,4}{ \node[draw,circle,inner sep=1pt,gray,fill] at (\x,2*\x) {};}
\foreach \x in {1,2,3,4,5}{ \node[draw,circle,inner sep=1pt,gray,fill] at (\x,2*\x-1) {};}
\foreach \x in {2,3,4,5}{ \node[draw,circle,inner sep=1pt,gray,fill] at (\x,2*\x-2) {};}
\foreach \x in {0,1,...,10}{ \node[draw,circle,inner sep=1pt,gray,fill] at (\x,0) {};}
\foreach \x in {1,2,3,4}{ \node[draw,circle,inner sep=1pt,gray,fill] at (\x,1) {};}
\foreach \x in {5,6,...,10}{ \node[draw,circle,inner sep=1pt,gray,fill] at (\x,1) {};}
\foreach \x in {2,3,4,5}{ \node[draw,circle,inner sep=1pt,gray,fill] at (\x,2) {};}
\foreach \x in {6,7,...,10}{ \node[draw,circle,inner sep=1pt,gray,fill] at (\x,2) {};}
\foreach \x in {2,3,4,5}{ \node[draw,circle,inner sep=1pt,gray,fill] at (\x,3) {};}
\foreach \x in {6,7,...,10}{ \node[draw,circle,inner sep=1pt,gray,fill] at (\x,3) {};}
\foreach \x in {2,3,...,6}{ \node[draw,circle,inner sep=1pt,gray,fill] at (\x,4) {};}
\foreach \x in {7,8,...,10}{ \node[draw,circle,inner sep=1pt,gray,fill] at (\x,4) {};}
\foreach \x in {3,4,...,10}{ \node[draw,circle,inner sep=1pt,gray,fill] at (\x,5) {};}
\foreach \x in {3,4,...,10}{ \node[draw,circle,inner sep=1pt,gray,fill] at (\x,6) {};}
\foreach \x in {4,5,...,10}{ \node[draw,circle,inner sep=1pt,gray,fill] at (\x,7) {};}
\foreach \x in {4,5,...,10}{ \node[draw,circle,inner sep=1pt,gray,fill] at (\x,8) {};}
\foreach \x in {5,...,10}{ \node[draw,circle,inner sep=1pt,gray,fill] at (\x,9) {};}
\foreach \x in {5,6,...,10}{ \node[draw,circle,inner sep=1pt,gray,fill] at (\x,10) {};}
\foreach \x in {6,...,10}{ \node[draw,circle,inner sep=1pt,gray,fill] at (\x,11) {};}
\foreach \x in {6,...,10}{ \node[draw,circle,inner sep=1pt,gray,fill] at (\x,12) {};}
\foreach \x in {7,...,10}{ \node[draw,circle,inner sep=1pt,gray,fill] at (\x,13) {};}
\foreach \x in {7,...,10}{ \node[draw,circle,inner sep=1pt,gray,fill] at (\x,14) {};}
\end{tikzpicture}
\hspace{1cm}
\begin{tikzpicture}[scale=0.3]
\filldraw[fill=violet!20,draw=violet] (10,14) -- (10,.315) -- (2.565,.315) --(1.063, .916) -- (.975,1.033) --(7.458,14) -- (10,14) -- cycle; 
\filldraw[fill=gray!30,draw=gray] (10,14) -- (10,.6) -- (3,.6) --(2, 1) -- (1.4,1.8) --(7.5,14) -- (10,14) -- cycle; 
\filldraw[fill=blue!10,draw=blue!40] (10,14) -- (10,.83) -- (5.17,.83) --(2.67,1.83) -- (2.17,2.5) --(7.92,14) -- (10,14) -- cycle; 
\filldraw[fill=teal!10,draw=teal!40] (10,14) -- (10,1.8) -- (7.2,1.8) --(6.2,2.2) -- (3.8,5.4) --(8.1,14) -- (10,14) -- cycle; 
\draw[black, thick] (0,0) -- (10.1 ,0);
\draw[black, thick] (0,0) -- (7,14);
\draw[gray] (8,1) -- (10,1);
\draw[gray] (3,3) -- (8.5,14);
\draw[gray] (8,1) -- (9,3);
\draw[gray] (3,3) -- (9,3);
\draw[blue!60] (7,2) -- (10,2);
\draw[blue!60] (4,6) -- (8,14);
\draw[blue!60] (7,2) -- (9,6);
\draw[blue!60](4,6) -- (9,6);
\foreach \x in {1,2,3,4}{ \node[draw,circle,inner sep=1pt,gray,fill] at (\x,2*\x) {};}
\foreach \x in {1,2,3,4,5}{ \node[draw,circle,inner sep=1pt,gray,fill] at (\x,2*\x-1) {};}
\foreach \x in {2,3,4,5}{ \node[draw,circle,inner sep=1pt,gray,fill] at (\x,2*\x-2) {};}
\foreach \x in {0,1,...,10}{ \node[draw,circle,inner sep=1pt,gray,fill] at (\x,0) {};}
\foreach \x in {1,2,3,4}{ \node[draw,circle,inner sep=1pt,gray,fill] at (\x,1) {};}
\foreach \x in {5,6,...,10}{ \node[draw,circle,inner sep=1pt,gray,fill] at (\x,1) {};}
\foreach \x in {2,3,4,5}{ \node[draw,circle,inner sep=1pt,gray,fill] at (\x,2) {};}
\foreach \x in {6,7,...,10}{ \node[draw,circle,inner sep=1pt,gray,fill] at (\x,2) {};}
\foreach \x in {2,3,4,5}{ \node[draw,circle,inner sep=1pt,gray,fill] at (\x,3) {};}
\foreach \x in {6,7,...,10}{ \node[draw,circle,inner sep=1pt,gray,fill] at (\x,3) {};}
\foreach \x in {2,3,...,6}{ \node[draw,circle,inner sep=1pt,gray,fill] at (\x,4) {};}
\foreach \x in {7,8,...,10}{ \node[draw,circle,inner sep=1pt,gray,fill] at (\x,4) {};}
\foreach \x in {3,4,...,10}{ \node[draw,circle,inner sep=1pt,gray,fill] at (\x,5) {};}
\foreach \x in {3,4,...,10}{ \node[draw,circle,inner sep=1pt,gray,fill] at (\x,6) {};}
\foreach \x in {4,5,...,10}{ \node[draw,circle,inner sep=1pt,gray,fill] at (\x,7) {};}
\foreach \x in {4,5,...,10}{ \node[draw,circle,inner sep=1pt,gray,fill] at (\x,8) {};}
\foreach \x in {5,...,10}{ \node[draw,circle,inner sep=1pt,gray,fill] at (\x,9) {};}
\foreach \x in {5,6,...,10}{ \node[draw,circle,inner sep=1pt,gray,fill] at (\x,10) {};}
\foreach \x in {6,...,10}{ \node[draw,circle,inner sep=1pt,gray,fill] at (\x,11) {};}
\foreach \x in {6,...,10}{ \node[draw,circle,inner sep=1pt,gray,fill] at (\x,12) {};}
\foreach \x in {7,...,10}{ \node[draw,circle,inner sep=1pt,gray,fill] at (\x,13) {};}
\foreach \x in {7,...,10}{ \node[draw,circle,inner sep=1pt,gray,fill] at (\x,14) {};}
\end{tikzpicture}
\captionsetup{margin=0in,width=4.5in,font=small,justification=centering}
\caption{$\Newt{(I_1)}$ with $\Newt{(I_2)}$ and $c_1\Newt{(I_1)}+c_2\Newt{(I_2)}$ for $(c_1,c_2) \in \{(\frac27 ,\frac1{34}), (\frac15,\frac15), (\frac12, \frac16), (\frac15,\frac45)\}$.}
\label{fig:NewtMin}
\end{figure}

Now, by part (4) of Lemma~\ref{lem:facedataop}, we can compute face data of $Q_\bc := c_1\Newt I_1+c_2\Newt I_2$.
$$\begin{array}{ccc}
\text{Face of }I_1 & \text{Corr.\ vertex of }c_2\Newt I_2 & \text{Face datum of }Q_\bc\\
(y, c_1) & (7c_2,2c_2) & (y, c_1 + 2c_2)\\
(2x+5y,21c_1) & (7c_2,2c_2) & (2x+5y, 21c_1 + 24c_2)\\
(2x-y,3c_1) & (4c_2,6c_2) & (2x-y,c_1+2c_2)\\
\\
\text{Face of }I_2 & \text{Corr.\ vertex of }c_1\Newt I_1 & \text{Face datum of }Q_\bc\\
(y,2c_2) & (8c_1,c_1) & (y, c_1+2c_2)\\
(4x+3y,34c_2) & (3c_1,3c_1) & (4x+3y, 21c_1+34c_2)\\
(2x-y,2c_2) & (3c_1,3c_1) & (2x-y, 3c_1+2c_2)\\
\end{array}$$
Removing the redundant conditions, we have that $Q_\bc$ is defined by face data $$\{(y,c_1+2c_2),(2x+5y,21c_1+24c_2),(4x+3y,21c_1+34c_2),(2x-y,3c_1+2c_2)\}.$$  
In second graph in Figure \ref{fig:NewtMin}, we illustrate the Minkowski sum $c_1\Newt{(I_1)}+c_2\Newt{(I_2)}$ for three different pairs $(c_1,c_2)$: $(\frac27, \frac1{34})$ in violet,  $(\frac15, \frac15)$ in gray,  $(\frac12, \frac16)$ in blue, and  $(\frac15, \frac45)$ in green.
\

Set $\bw = (1,1)$. We consider examples BH-ideals which arise as a mixed multiplier ideal and Fujino's non-LC ideal respectively. In particular, set $J:=J_{I_1,I_2}(\bw,\bc)=\left\langle x^\ba \mid \ba\in P_\bc\right\rangle$, where $P_\bc = Q_\bc-(1,1)$.  Now $P_\bc$ is defined by face data $$\{(y,c_1+2c_2-1),(2x+5y,21c_1+24c_2-7),(4x+3y,21c_1+34c_2-7),(2x-y,3c_1+2c_2-1)\}.$$ Let $\bd = (-2,-1).$
By Theorem~\ref{thm:polyhedra}, $J$ is fixed by the differential operator
$\delta_\bc = x^{-2}y^{-1} H_\bd(\theta)f_\bc(\theta)$, where 
$$f_\bc = (y,\lceil c_1+2c_2-1\rceil)!\cdot(2x+5y, \lceil 21c_1+24c_2+1\rceil)!\cdot (4x+3y,\lceil 21c_1+34c_2+3\rceil)!\cdot(2x-y,\lceil 3c_1+2c_2+1\rceil)!.$$
Also $J^\circ:=J^\circ(\bw,\bc)=\left\langle x^\ba\mid \ba\in\inter P_\bc\right\rangle$ is fixed by $\delta^\circ_\bc = x^{-2}y^{-1} H_\bd(\theta)g_\bc(\theta)$, where
$$g_\bc = (y,\lfloor c_1+2c_2\rfloor)!\cdot(2x+5y, \lfloor 21c_1+24c_2+2\rfloor)!\cdot (4x+3y,\lfloor 21c_1+34c_2+4\rfloor)!\cdot(2x-y,\lfloor 3c_1+2c_2+2\rfloor)!$$
Now $J^\circ$ is the unit ideal if $1$ is in the image of $\delta^\circ_\bc$, which occurs exactly when $\delta^\circ_\bc(x^2y)\neq 0$.  This happens when 
$$0\neq g(2,1) =(1,\lfloor c_1+2c_2\rfloor)!\cdot(9, \lfloor 21c_1+24c_2+2\rfloor)!\cdot (11,\lfloor 21c_1+34c_2+4\rfloor)!\cdot(3,\lfloor 3c_1+2c_2+2\rfloor)!,$$
which is equivalent to the system of inequalities
\begin{align*}
    c_1+2c_2 & < 1\\
    21c_1+24c_2 & < 7\\
    21c_1+34c_2 & < 7\\
    3c_1+2c_2 & < 1
\end{align*}
The third inequality above implies the other three, so $J^\circ$ is the unit ideal for all $\bc=(c_1,c_2)$ with $21c_1+34c_2<7$.  A similar computation shows that $J$ is the unit ideal for all $\bc=(c_1,c_2)$ with $21c_1+34c_2\leq 7$.

\

More generally, for any monomial $x^a y^b \in J^\circ(\bw,\bc)$ if and only if 
$$0\neq g(a+2,b+1) = (b+1,\lfloor c_1+2c_2\rfloor)!\cdot(2a+5b+9, \lfloor 21c_1+24c_2+2\rfloor)!\cdot (4a+3b+11,\lfloor 21c_1+34c_2+4\rfloor)!\cdot(2a-b+3,\lfloor 3c_1+2c_2+2\rfloor)!,
$$
equivalently,
\begin{align*}
c_1+2c_2 & < b+1\\
21c_1 + 24c_2 & < 2a+5b+7\\
    21c_1+34c_2 &< 4a+3b+7\\
    3c_1+2c_2 &<2a-b+1.
\end{align*}

Solving the above system of inequalities, we can determine regions in $\RR_{\geq 0}^2$ where the various monomials of $R$, must be in  $J^o(\bw,\bc)$.   Looking at the intersections of these regions, we can determine the consistency regions for $J^o(\bw,\bc)$ and we plot some of them in Figure \ref{fig:LCWalt},
for \[\{(c_1,c_2) \mid c_1 \geq 0, c_2 \geq 0, c_1 \leq \frac23, c_2 \leq \frac23\}.\]

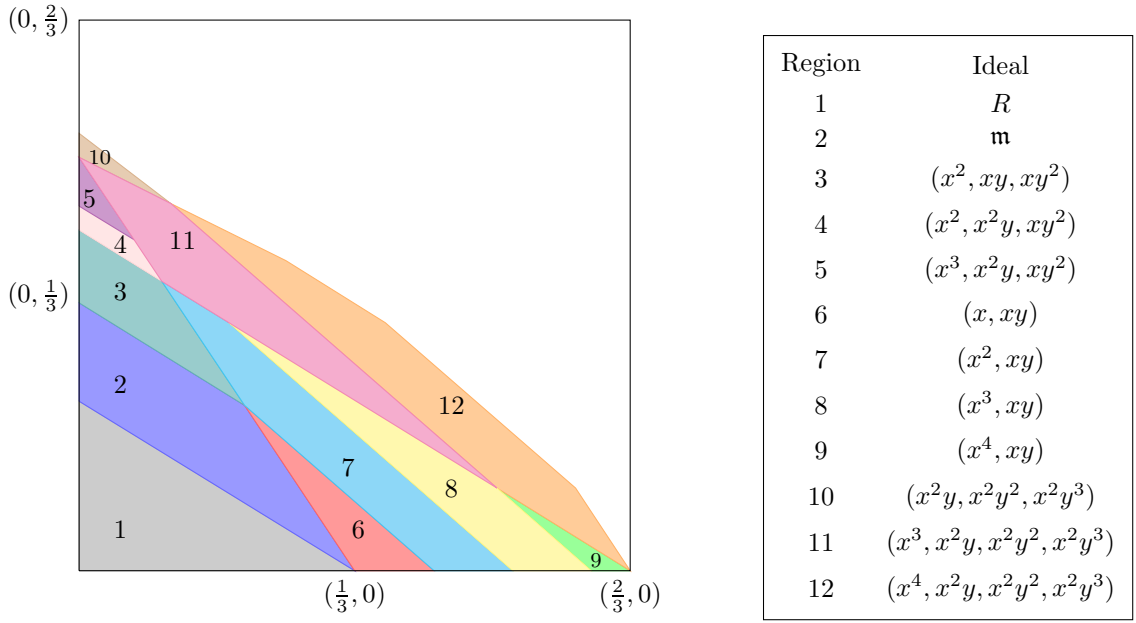
\begin{figure}[h]
    \centering
    \begin{tikzpicture}[scale=0.55]
 \filldraw[gray!40,draw=gray!60] (0,0)  --(0,4.1)  -- (6.66,0) --(0,0);
 \draw (1,1) node {1};
\filldraw[blue!40,draw=blue!60] (0,4.1)  -- (0,6.48) --(4,4) -- (6.66,0) --(0,4.1);
 \draw (1,4.5) node {2};
 \filldraw[teal!40,draw=teal!60] (0,6.48)  -- (0,8.24) -- (2,7) --(4,4) --(0,6.48);
 \draw (1,6.75) node {3};
 \filldraw[pink!40,draw=pink!60] (0,8.24) -- (0,8.82) -- (1.34, 8)-- (2,7)  -- (0,8.24);
 \draw (1,7.9) node {4};
 \filldraw[violet!40,draw=violet!60] (0,8.82) -- (0,10) --(1.34,8) --(0,8.82);
  \draw (.25,9) node {5};
 \filldraw[red!40,draw=red!60] (6.66,0) -- (8.58, 0) -- (4,4) -- (6.66,0);
  \draw (6.75,1) node {6}; 
 \filldraw[green!40,draw=green!60]  (13.32,0) -- (10.1,2) -- (12.38,0) -- (13.32,0);
 \draw (12.5,.25) node  {\footnotesize 9}; 
\filldraw[orange!40,draw=orange!60] (2.22,8.88) -- (5,7.5) -- (7.4,6)-- (12,2) -- (13.32,0) -- (10.1,2) --  (2.22,8.88);
\draw (9,4) node{12}; 
\filldraw[brown!40,draw=brown!60] (2.22,8.88) -- (0,10) --(0,10.58) -- (2.22,8.88);
\draw (.5,10) node {\footnotesize 10}; 
\filldraw[cyan!40,draw=cyan!60] (8.58, 0) -- (10.48,0) -- (3.62, 6) -- (2,7) -- (4,4) -- (8.58, 0);
\draw (6.5,2.5) node {7}; 
\filldraw[yellow!40,draw=yellow!60]  (3.62,6)  -- (10.48,0) -- (12.38,0) -- (10.1,2)  -- (3.62,6);
\draw (9,2) node {8}; 
\filldraw[magenta!40,draw=magenta!60]  (10.1,2)  -- (2,7) -- (0,10) --  (2.22,8.88)  -- (10.1,2);
\draw (2.5,8) node {11}; 
\draw[black] (0,0) -- (13.32 ,0) -- (13.32,13.32) --(0,13.32) -- (0,0);
\draw (6.66,0) node [below] {$(\frac13,0)$};
\draw (0,6.66) node [left] {$(0,\frac13)$};
\draw (13.32,0) node [below] {$(\frac23,0)$};
\draw (0,13.32) node [left] {$(0,\frac23)$};
 \end{tikzpicture}
\hspace{1cm}
\begin{tikzpicture}
    \tikzstyle{every column}=[draw]
\matrix [draw=black,column sep=.1cm]
{\node {Region}; & \node{Ideal};  \\
\node {1}; & \node{$R$};  \\
\node {2}; & \node{$\fm$};  \\
\node {3}; & \node{$(x^2,xy,xy^2)$};  \\
\node {4}; & \node{$(x^2,x^2y,xy^2$)};  \\
\node {5}; & \node{$(x^3,x^2y,xy^2)$};  \\
\node {6}; & \node{$(x,xy)$};  \\
\node {7}; & \node{$(x^2,xy)$};  \\
\node {8}; & \node{$(x^3,xy)$}; \\
\node {9}; & \node{$(x^4,xy)$}; \\
\node {10}; & \node{$(x^2y,x^2y^2,x^2y^3)$}; \\
\node {11}; & \node{$(x^3,x^2y,x^2y^2,x^2y^3)$}; \\
\node {12}; & \node{$(x^4,x^2y,x^2y^2,x^2y^3)$}; \\
};
\end{tikzpicture}
\captionsetup{margin=0in,width=4.5in,font=small,justification=centering}
    \caption{Some consistency regions for $J^o(\bw,\bc)$ for $(c_1,c_2) \in [0,\frac23] \times [0,\frac23]$}
   \label{fig:LCWalt}
\end{figure}

\end{xmp}

\end{document}